\documentclass{amsart}

\usepackage[T1]{fontenc}
\usepackage{latexsym,amssymb,amsmath}
\usepackage[all,ps]{xy}

 \newcommand{\Galg}{\mathbf{G}}
 \newcommand{\Halg}{\mathbf{H}}
 
 \newcommand{\Balg}{\mathbf{B}}
 \newcommand{\cprod}{\centerdot}
 \newcommand{\Zalg}{\mathbf{Z}}
 
 \newcommand{\Ualg}{\mathbf{U}}
 \newcommand{\Palg}{\mathbf{P}}
 
 \newcommand{\Lalg}{\mathbf{L}}
 \newcommand{\Talg}{\mathbf{T}}

 \newcommand{\Xalg}{\mathbf{X}}

 \newcommand{\PSL}{\operatorname{PSL}}
 \newcommand{\Orth}{\operatorname{P\Omega}}
 \newcommand{\PSp}{\operatorname{PSp}}
 
 \newcommand{\Out}{\operatorname{Out}}
 \newcommand{\Stab}{\operatorname{Stab}}

 \newcommand{\Ind}{\operatorname{Ind}}

 \newcommand{\Res}{\operatorname{Res}}
 \newcommand{\Nn}{\operatorname{N}}

 \newcommand{\Cen}{\operatorname{C}}

 \newcommand{\cyc}[1]{\langle\,#1\,\rangle}

 \newcommand{\C}{\mathbb{C}}
 \newcommand{\F}{\mathbb{F}}
 
 \newcommand{\N}{\mathbb{N}}
 \newcommand{\rk}{\operatorname{rk}}
 
 \newcommand{\Z}{\mathbb{Z}}

 \newcommand{\Irr}{\operatorname{Irr}}

 \newcommand{\cal}[1]{\mathcal #1}

\newtheorem{theorem}{Theorem}[section] 
\newtheorem{lemma}[theorem]{Lemma}     
\newtheorem{corollary}[theorem]{Corollary}
\newtheorem{proposition}[theorem]{Proposition}
\newtheorem{hypothese}[theorem]{Hypothesis}
\newtheorem{remark}[theorem]{Remark}

\title[On equivariant bijections relative to the defining characteristic]
 {On equivariant bijections relative to the defining characteristic} 

\author{Olivier Brunat}

\address{OB: Ruhr-Universit\"at Bochum\\
Fakult\"at f\"ur Mathematik\\
Raum NA 2/33\\
D-44780 Bochum\\}
\email{Olivier.Brunat@ruhr-uni-bochum.de}

\author{Frank Himstedt}
\address{FH: Technische Universit\"at M\"unchen\\
Zentrum Mathematik M11\\
Boltzmannstr.~3\\
D-85748 Garching\\}
\email{himstedt@ma.tum.de}

\begin{document}
\begin{abstract}
This paper is a contribution to the general program introduced by
Isaacs, Malle and Navarro to prove the McKay conjecture in the
representation theory of finite groups. We develop new methods for
dealing with simple groups of Lie type in the defining
characteristic case. Using a general argument based on the
representation theory of connected reductive groups with disconnected
center, we show that the inductive McKay condition holds if the Schur
multiplier of the simple group has order $2$. As a consequence,  the
simple groups $\Orth_{2m+1}(p^n)$ and $\PSp_{2m}(p^n)$ are 
``good'' for $p>2$ and the simple groups $E_7(p^n)$ are ``good'' for  
$p>3$ in the sense of Isaacs, Malle and Navarro. We also describe the
action of the diagonal and field automorphisms on the semisimple and
the regular characters.   
\end{abstract}

\maketitle


\section{Introduction}
\label{intro}

Let $G$ be a finite group, $p$ a prime dividing the order of $G$ and
$P$ a Sylow $p$-subgroup of $G$. The McKay conjecture asserts that the
number $|\Irr_{p'}(G)|$ of irreducible complex characters of $G$ of
degree not divisible by $p$ coincides with the number  
$|\Irr_{p'}(N_G(P))|$ of irreducible complex characters of the
normalizer $N_G(P)$ of degree not divisible by~$p$.

In \cite{IMN}, Isaacs, Malle and Navarro reduced the proof of the McKay
conjecture to a question about finite simple groups. They were able to
prove that the McKay conjecture is true for all finite groups if every
finite non-abelian simple group is ``good'' for all prime numbers $p$;
see \cite[Section 10]{IMN} for a precise formulation.

Malle has shown that all simple groups not of Lie type and all simple
groups of Lie type with exceptional Schur multiplier are ``good'' for all
primes $p$ \cite{mallenotlie}. Furthermore,
Malle \cite{malleextuni}, \cite{malleheight0} and
Sp\"ath \cite{spaethdiss} proved important results for simple groups
$G$ of Lie type and primes $p$ different from the defining
characteristic. The case where $G$ is a 
simple group of Lie type and $p$ the defining characteristic was
considered by several authors. It was shown in \cite{IMN} that the
simple groups $\PSL_2(q)$, ${^2B}_2(2^{2n+1})$, ${^2G}_2(3^{2n+1})$ and
in \cite{HH1}, \cite{HH2} that the simple groups ${^2F}_4(2^{2n+1})$
and ${^3D}_4(2^n)$ are ``good'' for the defining characteristic. A uniform
treatment of simple groups of Lie type with trivial Schur multiplier
and cyclic outer automorphism group in the case of defining
characteristic was obtained in \cite{Br6}. In particular, the results
in \cite{Br6} include that the simple groups $G_2(q)$, $F_4(q)$ and
$E_8(q)$ are ``good'' for the defining characteristic.

In the present paper, we consider certain simple groups of Lie type
with Schur multiplier of order $2$ in the defining characteristic
case. More precisely, our main result is the following  

\begin{theorem}\label{main}
Let $\Galg$ be a simple and simply-connected algebraic
group defined over the finite field $\F_q$ of characteristic $p>0$ with
corresponding Frobenius map $F:\Galg\rightarrow\Galg$. Let $\Zalg$ be
the center of $\Galg$ and let $W$ denote its Weyl group (relative to an
$F$-stable maximal torus $\Talg$ contained in an $F$-stable Borel
subgroup $\Balg$ of $\Galg$). Suppose that  
\begin{itemize}
\item the finite group $X=\Galg^F/\operatorname{Z}(\Galg^F)$ is simple,
\item the group $\Galg^F$ is the universal cover of $X$,
\item the automorphism induced by $F$ on $W$ is trivial,
\item the prime $p$ is good for $\Galg$ and
\item the center $\Zalg$ has order $2$.
\end{itemize}
Then the finite simple group $X$ is ``good'' for the prime~$p$.
\end{theorem}

A crucial step in the proof of Theorem~\ref{main} is to show the
existence of $\mathcal{A}$-equivariant bijections between certain
subsets of $\Irr_{p'}(\Galg^F)$ and $\Irr_{p'}(\Balg^F)$, where~$\Balg$
is an $F$-stable Borel subgroup of $\Galg^F$ and $\mathcal{A}$ is a
group of outer automorphisms of~$\Galg^F$ stabilizing $\Balg^F$. 
We had to solve several problems which do not show up in
the case where the Schur multiplier is trivial:

First, we had to consider the action of the diagonal and field
automorphisms of~$\Galg^F$ on $\Irr_{p'}(\Galg^F)$.
Under the assumptions of Theorem~\ref{main}, 
the set $\Irr_{p'}(\Galg^F)$ is the set of semisimple characters of
$\Galg^F$, and these characters are up to signs the duals (with respect to
Alvis-Curtis duality) of 
the irreducible constituents of the Gelfand-Graev
characters of $\Galg^F$. 
When the center~$\Zalg$ of $\Galg$ is connected, we can label the
semisimple characters of $\Galg^F$ by the set of
semisimple conjugacy classes of $\Galg^{*F^*}$, where $(\Galg^*,F^*)$
is a pair dual to $(\Galg,F)$.  When the center of $\Galg$ is
not connected, there is a similar, but more complicated
parametrization. 
It depends on some additional parameters whose
choice is not canonical, and it is \emph{a priori} a difficult problem
to describe the action of the automorphisms of $\Galg^F$ on these
characters with respect to this labelling.  To solve this problem 
(see Subsection~\ref{rk:jordan}), we use the theory of Gelfand-Graev
characters for connected reductive groups with disconnected center
developed by Digne-Lehrer-Michel in~\cite{DLM} and~\cite{DLM2}. 

Second, we had to find a suitable parametrization of
$\Irr_{p'}(\Balg^F)$. When the center~$\Zalg$ of $\Galg$ is connected, the set of
orbits of $\Talg^F$ on the linear characters of the unipotent radical
of $\Balg^F$ can be parametrized by subsets of the set of simple
roots, and Clifford theory leads to a particularly nice
parametrization of $\Irr_{p'}(\Balg^F)$. When~$\Zalg$ is not
connected, Clifford theory still applies, but the parametrization
becomes more complicated. To solve this problem, we had to introduce,
as above, additional parameters whose choice is not canonical.

Third, the above bijections have to be compatible with linear
characters of the center $\Zalg^F$ of $\Galg^F$. To prove the
existence of such bijections we use counting arguments based on the
norm map $N_{F'^m/F'}:\Zalg^{F'^m}\rightarrow\Zalg^{F'}$, where $m$ is
some positive integer and $F':\Galg\rightarrow\Galg$ a Frobenius
map inducing some field automorphism of~$\Galg^F$. One of the
difficulties we had to face is, that, in general, the norm map is not
surjective (because the center of $\Galg$ is not connected).  

Finally, in the situation of Theorem~\ref{main}, the cohomology
condition occurring in the definition of ``good'' becomes non-trivial
and has to be treated. In particular, we had to study extensions of
the characters of $p'$-degree of $\Galg^F$ and $\Balg^F$.

Some of our results on equivariant bijections are
also true in the more general context where the group of diagonal
automorphisms has prime order (possibly~$>2$). This might be helpful
in proving that the remaining simple groups of Lie type are ``good''.

We point out that we prove Theorem~\ref{main} using general methods
(essentially the theory of Deligne-Lusztig and the theory of
Gelfand-Graev characters for connected reductive groups with
disconnected center). Theorem~\ref{main} applies to finite
simple groups of Lie type $B_m$, $C_m$ and $E_7$ in the defining
characteristic. So we get as consequence:
\begin{corollary}
\label{cormain}
Let $p$ be a prime and $m,\,n$ positive integers. The following
simple groups are ``good'' for the prime $p$:
\begin{itemize}
\item[(a)] $\Orth_{2m+1}(p^n)$ if $p>2$ and $m \ge 2$,
\item[(b)] $\PSp_{2m}(p^n)$ if $p>2$ and $m \ge 2$,
\item[(c)] $E_7(p^n)$ if $p>3$.
\end{itemize}
\end{corollary}
Independently of this work, J.~Maslowski obtains in his PhD thesis
partial results on the inductive McKay condition for classical groups
in defining characteristic. Note that his approach relies on the
natural matrix representations of these groups and is completely
different from ours; see \cite{Maslowski}. 

This paper is organized as follows. In Section~\ref{nota}, we introduce
the notation and general setup.  In Sections~\ref{partie1} and
\ref{partborel}, we describe a parametrization of the sets of
irreducible characters of $p'$-degree of $\Galg^F$ and $\Balg^F$,
respectively, and study the action of field and diagonal automorphisms
on these sets of characters.  Section~\ref{partie2} is at the heart of
this paper. We prove the existence of bijections between the sets of
irreducible characters of $p'$-degree of $\Galg^F$ and $\Balg^F$ which
are equivariant with respect to field and diagonal automorphisms, and
which are compatible with the central characters of $\Zalg^F$. In
Section~\ref{indcond}, we prove our main result. 
A large part of this section is devoted to the proof of the cohomology
condition.


\section{Notation and Setup}
\label{nota}

In this section, we introduce the general setup and notation which
will be used throughout this paper.

\subsection{Group theoretical setup}
\label{grpsetup}
Le  $\Galg$ be a connected reductive group defined over a finite
field~$\F_q$ of characteristic $p>0$ with $q$ elements and
corresponding Frobenius map $F:\Galg\rightarrow\Galg$. 
We do not assume that the center $\Zalg$ of $\Galg$ is connected. As a
general reference for the representation theory of finite groups of
Lie type with disconnected center, we refer to \cite{BonnafeSLn}.

The algebraic group $\Galg$ can be embedded in a connected reductive
group $\widetilde{\Galg}$ with an $\F_q$-rational structure obtained by
extending $F$ to $\widetilde{\Galg}$, such that the center
of~$\widetilde{\Galg}$ is connected and the groups $\widetilde{\Galg}$
and $\Galg$ have the same derived subgroup (see for
example \cite[p.~139]{DM} for this construction). Let $\Talg$ be an
$F$-stable maximal torus of~$\Galg$ contained in an $F$-stable Borel
subgroup $\Balg$ of~$\Galg$ and write~$\widetilde{\Talg}$ for the
unique $F$-stable torus of $\widetilde{\Galg}$ containing $\Talg$. 
Fix an $F$-stable Borel subgroup $\widetilde{\Balg}$ of
$\widetilde{\Galg}$ containing~$\widetilde{\Talg}$ and~$\Balg$.
Write $\Ualg$ for the unipotent radical of $\Balg$ and let
$\Delta=\{\alpha_i\,|\,i\in I\}$ be the set of simple roots defined by
$\Balg$ and $\Talg$, and let $\Phi$ be the root system of $\Galg$ relative
to~$\Talg$. We write $\Phi^+$ for the set of positive roots defined by
$\Balg$.
Note that $\Ualg$ is the unipotent radical of
$\widetilde{\Balg}$. Furthermore, $\Phi$ can be identified with the
root system of $\widetilde{\Galg}$ relative to $\widetilde{\Talg}$, and
$\Delta$ can be identified with the set of simple roots of
$\widetilde{\Galg}$ defined by $\widetilde{\Talg}$ and~$\widetilde{\Balg}$.
Let $W \simeq \operatorname{N}_{\Galg}(\Talg)/\Talg$ be the Weyl group
of $\Galg$. So, $W$~is isomorphic to the Weyl group of $\widetilde{\Galg}$.
Note that, since $\Talg$ and $\widetilde{\Talg}$ are $F$-stable, 
$F$~induces an automorphism of $W$.
Throughout this paper, we assume that $F$ acts trivially on~$W$.

For $\alpha\in\Phi$, we write $\Xalg_{\alpha}$ for
the unipotent subgroup of $\Galg$ corresponding to $\alpha$. That is,
$\Xalg_{\alpha}$ is the minimal non-trivial closed unipotent subgroup
of $\Galg$ normalized by $\Talg$, such that $\Talg$ acts on
$\Xalg_{\alpha}$ by~$\alpha$, where $\alpha$ is viewed as a character of
$\Talg$.
 Recall that
$\Xalg_{\alpha}$ and $(\overline{\F}_p,+)$ are isomorphic as algebraic
groups. Fix an automorphism 
$x_{\alpha}:\overline{\F}_p\rightarrow\Xalg_{\alpha}$ for every
$\alpha\in\Phi$. Since $F$ acts trivially on $W$, we can choose 
it in such a way that for every $t\in\overline{\F}_p$, we have
$^Fx_{\alpha}(t)=x_{\alpha}(t^q)$. In particular, $\Xalg_{\alpha}$ is
$F$-stable.
We have $\Ualg=\prod\limits_{\alpha\in\Phi^+}\Xalg_{\alpha}$.
Put 
\begin{equation}\label{eq:etoile}
\Ualg_0=\prod\limits_{\alpha\in\Phi^+,\
\alpha\not\in\Delta}\Xalg_{\alpha}.
\end{equation}
Then $\Ualg_0$ is a normal
subgroup of $\Ualg$ and the quotient $\Ualg_1=\Ualg/\Ualg_0$ is abelian.
Note that $\Ualg_0$ is the derived subgroup of $\Ualg$;
see~\cite[14.17]{DM}. 
Moreover, there is an $F$- and $\Talg$-equivariant isomorphism of
algebraic varieties 
\begin{equation}\label{eq:etoilequo}
\Ualg_1\simeq \prod\limits_{\alpha\in\Delta}\Xalg_{\alpha},
\end{equation}
and in the following, we will identify the left and right hand side of
(\ref{eq:etoilequo}) via this isomorphism. Note that the groups
$\Ualg$ and $\Ualg_0$ are $F$-rational and
\begin{equation}\label{eq:etoileFrob}
\Ualg_1^F=\prod_{\alpha\in\Delta}\Xalg_{\alpha}^{F}.
\end{equation}

\subsection{Character theoretical notation}\label{charnota}

For a finite group $H$, we write $\Irr(H)$ for the set of complex
irreducible characters of $H$ and ${\cyc{\cdot, \cdot }_H}$ for the
usual scalar product on the space of class functions.
If $\zeta$ is an irreducible character of a normal subgroup $N$ of
$H$, we define $\Irr(H|\zeta) := \{\chi \in \Irr(H) \,| \,
{\cyc{\Res^H_N(\chi), \zeta}_N} \neq 0\}$. Note that, if $N$ is
central in $H$, then $\chi \in \Irr(H|\zeta)$ if and only if
$\Res^H_N(\chi)$ is a multiple of~$\zeta$. 

Furthermore, let
$\Irr_{p'}(H)$ be the set of all $\chi \in \Irr(H)$ such that
$\chi(1)$ is not divisible by $p$, and similarly, 
$\Irr_{p'}(H|\zeta)$ the set of all $\chi \in \Irr(H|\zeta)$ such that  
$\chi(1)$ is not divisible by $p$.

\subsection{Semisimple and regular characters}\label{regsemi}

Let $(\Galg^*,F^*)$ be a pair dual to $(\Galg,F)$ and let
$(\widetilde{\Galg}^*,F^*)$ be a pair dual to $(\widetilde{\Galg},F)$
in the sense of~\cite[Section 4.3]{Carter2}. The natural embedding
$i:\Galg\rightarrow\widetilde{\Galg}$ induces a surjective homomorphism
$i^*:\widetilde{\Galg}^*\rightarrow\Galg^*$ commuting with $F^*$,
which induces a surjective homomorphism 
$i^*:\widetilde{\Galg}^{*F^*}\rightarrow\Galg^{*F^*}$,
see \cite[2.D, 2.F]{BonnafeSLn}. 
Let $\Talg^*$ be an $F^*$-stable maximal torus of $\Galg^*$ in duality
with~$\Talg$. 
Note that there is a natural anti-isomorphism between $W$ and the Weyl
group of $\Galg^*$ relative to $\Talg^*$, and we will identify the sets
of elements of these two Weyl groups via this anti-isomorphism.  For
$w\in W$, we write $\Talg_w$ for a $F$-stable maximal torus of $\Galg$
obtained from $\Talg$ by twisting with $w$. Write $\Talg_w^*$ for an
maximal torus in duality with $\Talg_w$.  Recall that, for $w\in W$,
there is an isomorphism $\Talg_w^{*F^*}\rightarrow\Irr(\Talg_w^F)$.

For $s\in\Talg_w^{*F^*}$, we can define the corresponding 
Deligne-Lusztig character $R_{\Talg_w}^{\Galg}(s)$ as follows.
Using the above isomorphism, we associate to $s\in\Talg_w^{*F^*}$ the linear 
character $\theta_s\in\Irr(\Talg_w^F)$ and put
$R_{\Talg_w}^{\Galg}(s)=R_{\Talg_w}^{\Galg}(\theta_s)$, where
$R_{\Talg_w}^{\Galg}(\theta_s)$ is the Deligne-Lusztig character
corresponding to $\theta_s\in\Irr(\Talg_w^F)$. For more details on 
the construction and properties of Deligne-Lusztig characters, we
refer to~\cite[Section 7]{Carter2} or~\cite{DM}.

For a semisimple element $s  \in \Galg^{*F^*}$, let
$W^\circ(s) \subseteq W$ be the Weyl group
of~$C_{\Galg^*}^\circ(s)$. We define  
\begin{eqnarray}
\rho_{s}&=&\displaystyle{\frac{
\label{eq:rhos}
1}{|W^\circ({s})|}
\sum_{w\in W^\circ(s)}R_{{\Talg}_w}^{{\Galg}}
({s}),}\\
\chi_{{s}}&=&\displaystyle{\frac{\varepsilon_{{\Galg}} \label{eq:chis}
\varepsilon_{\Cen_{{\Galg}^*}^\circ(s)}}{|W^\circ({s})|}
\sum_{w\in W^\circ(s)}\varepsilon(w)R_{{\Talg}_w}^{
{\Galg}}({s}),}
\end{eqnarray}
where $\varepsilon$ is the sign character of $W$ and
$\varepsilon_{{\Galg}}=(-1)^{\rk_{\F_q}({\Galg
})}$. Here $\rk_{\F_q}({\Galg})$ is the $\F_q$-rank of
${\Galg}$; see~\cite[8.3]{DM}. Note that $\chi_s$ and $\rho_s$ only 
depend on the semisimple class of $s$ in $\Galg^{*F^*}$ and
that the class functions $\rho_s$ and $\chi_s$ are multiplicity free 
characters of $\Galg^F$, see~\cite[15.11]{BonnafeSLn}.
Moreover, if $\widetilde{s}$ denotes a semisimple element of
$\widetilde{\Galg}^F$ such that $i^*(\widetilde{s})=s$, then we have
$$\chi_s=\Res_{\Galg^F}^{\widetilde{\Galg}^F}(\chi_{\widetilde{s}})
\quad\textrm{and}
\quad
\rho_s=\Res_{\Galg^F}^{\widetilde{\Galg}^F}(\rho_{\widetilde{s}}).$$
The irreducible
constituents of $\rho_s$ and $\chi_s$ are the so-called semisimple 
and the regular characters of $\Galg^F$, respectively.

To obtain a better understanding of these characters, we now describe the 
Gelfand-Graev characters of $\Galg^F$. 
Fix $\phi_0 \in \Irr(\Ualg_1^F)$ such that
$\phi_0|_{\Xalg_{\alpha}^F}$ is non-trivial for all 
$\alpha \in \Delta$. 
The corresponding linear
character of~$\Ualg^F$, obtained by inflation and also denoted by
$\phi_0$ in the following, is called regular.  
As explained in~\cite[14.28]{DM}, the set of $\Talg^F$-orbits
of regular characters of $\Ualg^F$ is in bijection with
$H^1(F,\Zalg)$ as follows.
For $z\in H^1(F,\Zalg)$, we choose $t_z\in\Talg$
such that $t_z^{-1}F(t_z)\in z$. Then the regular character
$\phi_z={^{t_z}\!\phi_0}$ of $\Ualg^F$ is a representative for the
corresponding $\Talg^F$-orbit. 
For $z\in H^1(F,\Zalg)$, we define the corresponding
Gelfand-Graev character
of $\Galg^F$ by $$\Gamma_z=\Ind_{\Ualg^F}^{\Galg^F}(\phi_z).$$
Thanks to~\cite[14.49]{DM}, the constituents of $\Gamma_z$ (for any
$z\in H^1(F,\Zalg)$)
are the regular characters 
of $\Galg^F$.
More precisely, 
for every $z\in H^1(F,\Zalg)$,
the multiplicity free~characters $\chi_s$ and $\Gamma_z$ have exactly one irreducible
constituent in common, denoted by $\chi_{s,z}$, such that
$$\Gamma_z=\sum_{s\in\cal S}\,\chi_{s,z},$$
where $\cal S$ is a set of representatives for the semisimple
conjugacy classes of $\Galg^{*F^*}$; see \cite[14.49]{DM}.
Let $D_{\Galg}$ be the Alvis-Curtis duality
functor~\cite[8.8]{DM}, defined 
for $g\in\Galg^F$ and $\chi\in\Z\Irr(\Galg^F)$ by
\begin{equation}\label{eq:dualite}
D_{\Galg}(\chi)(g)=\sum_{\Palg\supseteq\Balg}(-1)^{r(\Palg)}
R_{\Lalg}^{\Galg}\circ ^*\!\!R_{\Lalg}^{\Galg}(\chi)(g),
\end{equation}
where the summation is over the set of $F$-stable parabolic subgroups of
$\Galg$ containing $\Balg$ and where $\Lalg$ is an $F$-stable Levi
complement of $\Palg$, $r(\Palg)$ is the semisimple $\F_q$-rank of
$\Palg$, $R_{\Lalg}^{\Galg}$ denotes the Harish-Chandra induction and
$^*R_{\Lalg}^{\Galg}$ the Harish-Chandra restriction (i.e., the
adjoint functor of $R_{\Lalg}^{\Galg}$); see~\cite[4.6]{DM}. 
For every semisimple element $s\in\Galg^{*F}$ and $z\in H^1(F,\Zalg)$, define
$$\rho_{s,z}=\varepsilon_{\Galg}\varepsilon_{\Cen_{{\Galg}^*}^\circ(s)}
D_{\Galg}(\chi_{s,z}).$$
Note that the characters $\rho_{s,z}$ are the semisimple characters of
$\Galg^F$ and $\{\rho_{s,z}\,|\,z\in H^1(F,\Zalg)\}$ is the set of
constituents of $\rho_s$.
Moreover, we have
$$D_{\Galg}(\Gamma_z)=\sum_{s\in\cal
S}\varepsilon_{\Galg}\varepsilon_{\Cen_{{\Galg}^*}^\circ(s)}\rho_{s,z}.$$

Let $\widetilde{\Talg}$ be the maximal $F$-stable torus of
$\widetilde{\Galg}$ containing $\Talg$. Then the group
$\widetilde{\Talg}^F$ acts on $\Galg^F$ by conjugation. Note that the
induced outer automorphisms of $\Galg^F$ obtained in this
way are the diagonal automorphisms of $\Galg^F$. The
group generated by the diagonal automorphisms of $\Galg^F$ will be denoted 
by $D$ in the following. Write $\widetilde{\Galg}^F(s)$ for the
inertia subgroup of $\rho_{s,1}$ in $\widetilde{\Galg}^F$ and 
$A_{\Galg^{*}}(s):=\Cen_{\Galg^*}(s)/\Cen_{\Galg^*}(s)^{\circ}.$
Then $|\widetilde{\Galg}^F/\widetilde{\Galg}^F(s)| = 
|A_{\Galg^*}(s)^{F^*}|$, see \cite[11.E, 15.13]{BonnafeSLn}. In
particular, the character~$\rho_s$ has $|A_{\Galg^*}(s)^{F^*}|$
irreducible constituents and they form a single
$\widetilde{\Talg}^F$-orbit. 

\subsection{The norm map} \label{normmap}

Let $\Halg$ be an abelian algebraic group defined over $\F_q$ with
corresponding Frobenius map $F'$. For any positive integer $m$ 
we define
\begin{equation}
N_{F'^m/F'}:\Halg^{F'^m}\rightarrow\Halg^{F'},\,h\mapsto hF'(h)\cdots
F'^{m-1}(h).
\label{eq:norm}
\end{equation}
For a class function $\theta$ of $\Halg^{F'}$, we set 
$N_{F'^m/F'}^*(\theta):=\theta\circ N_{F'^m/F'}$. So, for
$\theta\in\Irr(\Halg^{F'})$, the map $N_{F'^m/F'}^*(\theta)$ is an
$F'$-stable irreducible character of $\Halg^{F'^m}$.
\begin{lemma}\label{surjnorm}
If $N_{F'^m/F'}$ is surjective, then 
$$\Irr_{F'}(\Halg^{F'^m})=\{N_{F'^m/F'}^*(\theta)\ |\
\theta\in\Irr(\Halg^{F'})\},$$
where $\Irr_{F'}(\Halg^{F'^m})$ is the set of $F'$-stable linear characters
of $\Halg^{F'^m}$. 
Moreover, for all generalized characters $\theta,\,\theta'\in\Z\Irr(\Halg^{F'})$, one has
$$
\cyc{N_{F'^m/F'}^*(\theta),N_{F'^m/F'}^*(\theta')}_{\Halg^{F'^m}}=\cyc{\theta,\theta'}_{\Halg^{F'}}.
$$
\end{lemma}
\begin{proof}
By~\cite[6.32]{isaacs}, one has $|\Irr_{F'}(\Halg^{F'^m})|=|\Halg^{F'}|$
and, if $N_{F'^m/F'}$ is surjective, then the map
$N_{F'^m/F'}^*: \Irr(\Halg^{F'})\rightarrow\Irr(\Halg^{F'^m}),\,\theta\mapsto\theta\circ
N_{F'^m/F'}$ is injective. Since
$N_{F'^m/F'}^*(\theta)\in\Irr_{F'}(\Halg^{F'^m})$, the first equality
follows.
Now, let $\theta,\,\theta'\in\Z\Irr(\Halg^{F'})$. Then
\begin{eqnarray*}
\cyc{N_{F'^m/F'}^*(\theta),N_{F'^m/F'}^*(\theta')}_{\Halg^{F'}}&=&
\frac{1}{|\Halg^{F'^m}|}\sum_{h\in \Halg^{F'^m}}\hspace{-0.1cm}\theta\circ
N_{F'^m/F'}(h)\,
\overline{\theta'\circ N_{F'^m/F'}(h)}\\
&=&
\frac{1}{|\Halg^{F'^m}|}
\sum_{k\in\Halg^{F'}}\sum_{h\in N_{F'^m/F'}^{-1}(k)}\theta(k)\,
\overline{\theta'(k)}\\
&=&
\frac{1}{|\Halg^{F'^m}|}
\sum_{k\in\Halg^{F'}}|N_{F'^m/F'}^{-1}(k)|\theta(k)\,
\overline{\theta'(k)}\\
&=&
\frac{1}{|\Halg^{F'^m}|}
\sum_{k\in\Halg^{F'}}\frac{|\Halg^{F'^m}|}{|\Halg^{F'}|}\theta(k)\,
\overline{\theta'(k)}\\
&=&
\frac{1}{|\Halg^{F'}|}
\sum_{k\in\Halg^{F'}}\theta(k)\,
\overline{\theta'(k)}\\
&=&\cyc{\theta,\theta'}_{\Halg^{F'}}.
\end{eqnarray*}
This yields the claim.
\end{proof}
Note that if $\Halg$ is connected, then $N_{F'^m/F'}$ is surjective.
Indeed, if $y\in\Halg^{F'}$, then the Lang-Steinberg theorem implies that
there is $x\in\Halg$ with $y=x^{-1}F'^m(x)$. Then the element
$x^{-1}F'(x)$ lies in $\Halg^{F'^m}$ and $N_{F'^m/F'}(x^{-1}F'(x))=y$.

\subsection{Semisimple characters and central characters} \label{semicentral}

As in Subsection~\ref{grpsetup}, let~$\Galg$ be a connected reductive
group defined over $\F_q$ (with Frobenius map $F$) and let
$(\Galg^*,F^*)$ be a pair dual to $(\Galg,F)$ as above. Note that for
every positive integer~$m$, the map $F^m$ is a Frobenius map on
$\Galg$ defining a rational structure over~$\F_{q^m}$, and
$(\Galg^*,F^{*m})$ is in duality with $(\Galg,F^m)$.  Moreover, if $s$
is an $F^*$-stable semisimple element of~$\Galg^*$ contained in an
$F^*$-stable maximal torus $\Talg^*$ of $\Galg^*$ and if
$(\Talg,\theta)$ is a pair in duality with $(\Talg^*,s)$, then
$(\Talg, N_{F^m/F}^*(\theta))$ is in duality with $(\Talg^*,s)$ with
respect to the Frobenius map $F^m$.

\begin{lemma}\label{resZ}
Fix a positive integer $m$ and let $s \in \Galg^{*F^*}$ be a
semisimple element contained in the $F^*$-stable maximal torus
$\Talg^*$. Let $\rho_{s}$ (resp.~$\rho_{s}^{[m]}$) be the
corresponding sum of semisimple irreducible characters of~$\Galg^F$
(resp.~of~$\Galg^{F^m}$).  If
$N_{F^m/F}:\Zalg^{F^m}\rightarrow\Zalg^F$ is surjective, then one has
$$\Res^{\Galg^{F^m}}_{\Zalg^{F^m}}(\rho_{s}^{[m]})=\rho_s^{[m]}(1)
\cdot N_{F^m/F}^*(\nu),$$ where $\nu$ is a linear character of
$\Zalg^F$ such that $\Res_{\Zalg^F}^{\Galg^F}(\rho_s)=\rho_s(1)\cdot
\nu$.
\end{lemma}
\begin{proof}
We use the notation from Subsections~\ref{regsemi}
and \ref{normmap}. By the construction of~$\widetilde{\Galg}$, we have
$\Zalg = Z(\Galg) \subseteq Z(\widetilde{\Galg})$ and so
$\Zalg^{F^m} \subseteq Z(\widetilde{\Galg})^{F^m} =
Z(\widetilde{\Galg}^{F^m})$, see 
also \cite[6.2]{BonnafeSLn}. Thus, there is some 
$\nu \in \Irr(\Zalg^{F^m})$ such that
$$\Res_{\Zalg^{F^m}}^{\Galg^{F^m}}(\rho_s^{[m]})
=
\Res_{\Zalg^{F^m}}^{\widetilde{\Galg}^{F^m}}(\rho_{\widetilde{s}}^{[m]})
= \rho_{\widetilde{s}}^{[m]}(1) \cdot \nu
= \rho_s^{[m]}(1) \cdot \nu,$$
where $\widetilde{s}$ in an $F^*$-stable element of
$\widetilde{\Galg}$ satisfying $i^*(\widetilde{s})=s$.
For $w \in W^\circ(s)$, let $\theta_s^{[m]}$ be the linear character of
$\Talg_w^{F^m}$ associated with $s$. Thanks to~\cite[9.D]{BonnafeSLn},
we have
$$\Res_{\Zalg^{F^m}}^{\Talg_w^{F^m}}(\theta_s^{[m]})=\nu.$$
Furthermore, $s$ is $F^*$-stable. Thus, $\theta_s^{[m]}$ is
$F$-stable (see the proof of~\cite[1.1]{Br6}), implying $\nu$ is
$F$-stable. Since $N_{F^m/F}:\Zalg^{F^m}\rightarrow\Zalg^F$ is surjective,
Lemma~\ref{surjnorm} implies there is a linear character $\nu_0$ of
$\Zalg^{F}$ satisfying $N_{F^m/F}^*(\nu_0)=\nu$.
Write~$\theta_s$ for the linear character of
$\Talg_w^{F}$ associated with $s$. 
By the remarks preceding
Lemma~\ref{resZ}, we have $N_{F^m/F}^*(\theta_s)=\theta_s^{[m]}$.
Let $z_0\in\Zalg^{F}$. Then there is $z\in\Zalg^{F^m}$ with
$N_{F^m/F}(z)=z_0$. It follows
\begin{eqnarray*}
\nu_0(z_0)&=&\nu_0\circ N_{F^m/F}(z)\\
&=&\nu(z)\\
&=&\Res_{Z^{F^m}}^{T_w^{F^m}}(\theta_s^{[m]})(z)\\
&=&\theta_s^{[m]}(z)\\
&=&\theta_s\circ N_{F^m/F}(z)\\
&=&\theta_s(z_0). 
\end{eqnarray*} 
In particular, we have $\Res_{\Zalg^F}^{\Talg_w^F}(\theta_s)=\nu_0$
and~\cite[9.D]{BonnafeSLn} implies

$$\Res_{\Zalg^F}^{\Galg^F}(\rho_s)=\rho_s(1)\cdot\nu_0,$$ as
required.
\end{proof}

\subsection{Central products}\label{cenprod}
We recall some general facts about characters of central products. If
$N$ is a normal subgroup of a finite group $G$, then we can associate
to every $G$-invariant irreducible character $\chi$ of $N$ an element
$[\chi]_{G/N}$ of the cohomology group $H^2(G/N,\C^{\times})$ of
$G/N$; see~\cite[11.7]{isaacs} for more details. If $G=HK$ is a 
central product with $Z = H \cap K$ and $\nu$ a linear character of
$Z$, then for $\chi_H\in\Irr(H|\nu)$ and $\chi_K\in\Irr(K|\nu)$, one
defines 
\begin{equation}\label{eq:dotproduct}
(\chi_H\cprod\chi_K)(hk)=\chi_H(h)\chi_K(k)
\end{equation}
for all $h\in H$ and $k\in K$. Note that $\chi_H\cprod\chi_K$ is
a well-defined irreducible character of $HK$ and every
irreducible character of $HK$ has this form; see \cite[5.1]{IMN}.


\section{Action of automorphisms on semisimple characters of finite
reductive groups}
\label{partie1}

Let $F':\Galg\rightarrow\Galg$ be a Frobenius map of $\Galg$ commuting
with $F$ such that $\Talg$ and~$\Balg$ are $F'$-stable. In particular,
for $\alpha\in\Phi$, $F'(\Xalg_{\alpha})$ is a non-trivial minimal
closed unipotent subgroup of $\Galg$ normalized by $\Talg$. Write
$F'(\alpha)\in\Phi$ for the corresponding root. Moreover, we suppose
that for every $\alpha \in \Phi$, there is a non-negative integer
$m\geq 0$ such that for all $t\in\overline{\F}_p$ we have
$F'(x_{\alpha}(t))=x_{F'(\alpha)}(t^{p^m})$.  
Note that $\Ualg$ and~$\Ualg_1$ are $F'$-stable
(because $F$ and $F'$ commute), so $F'(\Phi^+)=\Phi^+$
and $F'(\Delta)=\Delta$. 

In this section, we study the action of $F'$ on
$\Irr_{p'}(\Galg^F)$. In a first step, we show that 
$\Irr_{p'}(\Galg^F)$ is exactly the set of semisimple irreducible
characters of $\Galg^F$. Then, we determine the action of $F'$ on these
semisimple characters. As an intermediate result, we obtain the action
of $F'$ on the set of regular irreducible characters of $\Galg^F$.

\subsection{$F'$-stable linear characters of $\Ualg_1^F$}\label{secstable}

We assume the setup from Section~\ref{nota}. 
For $J\subseteq \Delta$, let
$\widetilde{\omega}_J$ be the set of characters of $\Ualg_1$ of the
form $\prod_{\alpha\in J}\phi_{\alpha}$ where $\phi_{\alpha}$ is a
non-trivial irreducible character of $\Xalg_{\alpha}^F$. 

\begin{lemma} 
Let $J$ be an $F'$-stable subset of $\Delta$. Then
$\widetilde{\omega}_J$ contains an $F'$-stable character.
\label{unistable}
\end{lemma}

\begin{proof} 
Let $\Lambda$ be the set of $F'$-orbits on $J$. For
$\lambda\in\Lambda$, fix $\alpha_{\lambda}\in\lambda$ and write
$r_{\lambda}$ for a non-negative integer such that
$F'^{r_{\lambda}}(\alpha_{\lambda})=\alpha_{\lambda}$. 
Then $F'^{r_\lambda}$ is an automorphism of $\Xalg_{\alpha_{\lambda}}^F$ 
(because $F$ and $F'$ commute).
Since $F'^{r_{\lambda}}(x_{\alpha_{\lambda}}(1))=x_{\alpha_{\lambda}}(1)$, 
it follows
from~\cite[Theorem (6.32)]{isaacs} that $\Xalg_{\alpha_{\lambda}}^F$
has an $F'$-stable character. Fix such a character
$\phi_{\lambda}\in\Irr(\Xalg_{\alpha_{\lambda}}^F)$.
For $0\leq i\leq r_{\lambda}-1$, the map
$F'^i:\Xalg_{\alpha_{\lambda}}^F\rightarrow\Xalg_{F'^i(\alpha_{\lambda})}^F$
is a group isomorphism, and hence induces a natural bijection 
$\Irr(\Xalg_{\alpha_\lambda}^F) \rightarrow \Irr(\Xalg_{F'^i(\alpha_\lambda)}^F)$. We
write $^{F'^i}\!\!\phi_{\lambda}$ for the image of $\phi_\lambda$
under this bijection. Therefore, the character
$$\phi=\prod_{\lambda\in\Lambda}\prod_{i=0}^{r_{\lambda}-1}
{}^{{F'}^i}\!\!\phi_{\alpha_{\lambda}}$$
is in $\widetilde{\omega}_J$ and is $F'$-stable.
\end{proof}

\begin{remark}
Note that Lemma~\ref{unistable} is also true for a non-split Frobenius
$F$.
\end{remark}

\subsection{Action on the semisimple and regular characters}

In this subsection, we study the action of~$F'$ on the set of
semisimple and the set of regular irreducible characters of $\Galg^F$.

In the following, we say that the prime $p$ is nonsingular if
\begin{itemize}
\item if $p=2$, then $\Galg$ has no simple component of type $B_n$,
$C_n$, $F_4$ or $G_2$.
\item if $p=3$, then $\Galg$ has no simple component of type $G_2$.
\end{itemize}

\begin{lemma}\label{parasemi}
We keep the notation of Subsection~\ref{grpsetup}.
If $p$ is a nonsingular prime for
$\Galg$, then the set
of irreducible $p'$-characters of $\Galg^F$ is the set of semisimple
irreducible characters of $\Galg^F$. 
\end{lemma}

\begin{proof}Let $(\widetilde{\Galg}, F)$ be as above. 
Note that $p$ is a nonsingular prime for $\Galg$ if and only if $p$ is a
nonsingular prime for $\widetilde{\Galg}$.
In particular, 
the simple components of $\widetilde{\Galg}^F$ are not one of
the groups $B_m(2)$, $G_2(2)$, $G_2(3)$, $F_4(2)$, $^2B_2(2)$,
$^2G_2(3)$, $^2F_4(2)$. Moreover, 
Since the center of $\widetilde{\Galg}$ is
connected, the proof of \cite[Lemma 5]{Br6} also holds for
$\widetilde{\Galg}^F$, and so
$$\Irr_{p'}(\widetilde{\Galg}^F)=\{\rho_{\widetilde{s}}\ | \
\widetilde{s}\in\widetilde{S}\},$$
where~$\widetilde{S}$ is a set of representatives for the semisimple
classes of~$\widetilde{\Galg}^{*F^*}$.
In particular, the
semisimple characters of $\Galg^F$, which are the constituents 
of the restriction of~$\rho_{\widetilde{s}}$, are $p'$-characters of
$\Galg^F$. 

Conversely, let $\widetilde{\chi}$ be a non semisimple irreducible
character of~$\widetilde{\Galg}^F$, that is $p$ divides
$\widetilde{\chi}(1)$, and let $\gamma$ be an irreducible constituent
of
$\Res_{\Galg^F}^{\widetilde{\Galg}^F}(\widetilde{\chi})$. 
By \cite[11.29]{isaacs},
we have $\widetilde{\chi}(1) = m \cdot \gamma(1)$ where $m$ is an
integer dividing $|\widetilde{\Galg}^F / \Galg^F|$. Since 
$|\widetilde{\Galg}^F / \Galg^F|$ is not divisible by $p$, it follows
that $\gamma(1)$ is a multiple of $p$. Thus
$$\Irr_{p'}(\Galg^F)=\{\rho_{s,z}\ |\ s\in\mathcal{S},\,z\in
H^1(F,\Zalg)\}.$$
\end{proof}

\begin{remark} 
Note that, in Lemma~\ref{parasemi} we do not need to suppose that $p$
is a good prime for $\Galg$. Moreover, if $\Galg$ is simple and $p$ is
singular for $\Galg$ (the prime $p$ is said singular for $\Galg$ if it is not
nonsingular), then the $p'$-characters
of $\Galg^F$ are well-known, see for example~\cite[Remark 2]{Br6}. In
particular, the groups~$B_m(2)$ are treated in~\cite{Cabanes}.
\end{remark}

Let $\mathcal{S}$ be a set of representatives for the
semisimple classes of $\Galg^{*F^*}$. Since $F'$
stabilizes~$\widetilde{\omega}_{\Delta}$, Lemma~\ref{unistable}
implies that there is an $F'$-stable character in
$\widetilde{\omega}_{\Delta}$. We fix such a character
$\phi_0\in\widetilde{\omega}_{\Delta}$ and use it for the construction
of the Gelfand-Graev characters as described in~Subsection~\ref{regsemi}. 

\begin{lemma}\label{ImgDL} 
For $w\in W$, let $\Talg_w$ denote the $F$-stable maximal
torus of $\Galg$ obtained from $\Talg$ by twisting with $w$.
Then for all semisimple elements $s$ of $\widetilde{\Galg}^F$
and $w\in W$, one has
$$F'\left(R_{{\Talg}_{w}}^{{\Galg}}({s})\right)=
R_{{\Talg}_{F'(w)}}^{{\Galg}}
\left(F'^{*-1}({s})\right).$$
\end{lemma}

\begin{proof}
In the proof of~\cite[Proposition 1]{Br6}, it is shown that
$$F'\left(R_{{\Talg}_{w}}^{{\Galg}}({s})\right)=
R_{F'({\Talg}_{w})}^{{\Galg}}
\left(F'^{*-1}({s})\right).$$
Since $F$ and $F'$ commute, the maximal torus $F'({\Talg}_w)$ is
$F$-stable. We claim that $F'({\Talg}_w)$ is obtained from $\Talg$ by
twisting with $F'(w)$. There is $x\in{\Galg}$ such that
\begin{equation} \label{eq:twist}
x^{-1}F(x)=n_w\quad\textrm{and}\quad{\Talg}_w =x{\Talg}x^{-1},
\end{equation}
where $n_w\in\Nn_{{\Galg}}({\Talg})$ and $n_w{\Talg}=w$. Let
$n_{F'(w)} := F'(n_w)$. Since $\Talg$ is $F'$-stable, 
$n_{F'(w)} \in N_\Galg(\Talg)$ and $n_{F'(w)} \Talg = F'(w)$
and equation (\ref{eq:twist}) implies
$$F'({\Talg}_w)=F'(x){\Talg}F'(x)^{-1},$$
and
$$F'(x)^{-1}F(F'(x))=F'(x^{-1}F(x))=F'(n_w)=n_{F'(w)}.$$
This yields the claim.
\end{proof}

\begin{theorem}\label{imageF}
We make the same assumptions as in Lemma~\ref{ImgDL}. Additionally, we
suppose that $F'$ acts trivially on the root system $\Phi$.
Let $\Ualg$ be the unipotent radical of $\Balg$. 
For all $s\in\mathcal{S}$ and  $z\in H^1(F,\Zalg)$, one has
$$F'(\rho_{s,z})=\rho_{F'^{*-1}(s),F'(z)}\quad \textrm{and}\quad 
F'(\chi_{s,z})=\chi_{F'^{*-1}(s),F'(z)}.$$
\end{theorem}

\begin{proof}
For $s\in\mathcal{S}$, let $A_s$ be the set of constituents of
$\chi_{s}$. For $z\in H^1(F,\Zalg)$, write $A_z$ for the set of
constituents of $\Gamma_z$. We have $$A_s\cap A_{z}=\{\chi_{s,z}\}.$$
Let $\phi_z$ be a regular linear character of $\Ualg^F$ such that
$\Gamma_z=\Ind_{\Ualg^F}^{\Galg^F}(\phi_z)$.
Since $F$ and $F'$ commute, we have $F'(\Ualg^F)=\Ualg^F$. It follows
$$F'(\Gamma_z)=\Ind_{\Ualg^F}^{\Galg^F}(F'(\phi_z)).$$
Let $t_z\in\Talg$ such that $t_z^{-1}F(t_z)\in z$. Then, for
$u\in\Ualg^F$, one has
\begin{eqnarray*}
F'(\phi_z)(u)&=&\displaystyle{\phi_0(t_z^{-1}F'^{-1}(u)t_z)}\\
&=&\displaystyle{\phi_0(F'^{-1}(F'(t_z)^{-1}uF'(t_z)))}\\
&=&\displaystyle{^{F'(t_z)}\phi_0(u),}
\end{eqnarray*}
because $F'(\phi_0)=\phi_0$. Furthermore,
$$F'(t_z)^{-1}F(F'(t_z))=F'(t_z^{-1}F(t_z)) \in F'(z).$$
Thus, $^{F'(t_z)}\phi_0={^{t_{F'(z)}}\phi_0}$ and
\begin{equation}\label{eqgamma}
F'(\Gamma_z)=\Gamma_{F'(z)}.
\end{equation}
Now, 
Lemma~\ref{ImgDL} implies
$$F'(\chi_{{s}})=\frac{\varepsilon_{{\Galg}}
\varepsilon_{\Cen^\circ_{{\Galg}^*}({s})}}{|W^\circ(s)|}
\sum_{w\in
W^\circ({s})}\varepsilon(w)R_{{\Talg}_{F'(w)}}^{
{\Galg}}(F'^{*-1}({s})).$$
Since
$({\Talg}_{F'(w)})^*={\Talg}_{F'^{*-1}(w)}^*$
(see~\cite[4.3.2]{Carter2}), it follows
\begin{equation}
F'(\chi_s)=\chi_{F'^{*-1}(s)}.
\label{eqchi}
\end{equation}
Relations~(\ref{eqgamma}) and~(\ref{eqchi}) say that
$$F'(A_s)=A_{F'^{*-1}(s)}\quad\textrm{and}\quad
F'(A_z)=A_{F'(z)}.$$
Therefore

\begin{eqnarray*}
\{F'(\chi_{s,z})\}&=&\displaystyle{F'(\{\chi_{s,z}\})=F'(A_s\cap
A_z)}\\
&=&\displaystyle{F'(A_s)\cap
F'(A_z)=A_{F'^{*-1}(s)}\cap
A_{F'(s)}}\\
&=&\displaystyle{\{\chi_{F'^{*-1}(s),F'(z)}\}.}
\end{eqnarray*}
Thus 
\begin{equation}
F'(\chi_{s,z})=\chi_{F'^{*-1}(s),F'(z)}.
\label{eqregulier}
\end{equation}

Since $F'$ acts trivially on the set of roots, $F'$ stabilizes every
$F$-stable parabolic subgroup $\Palg$ containing $\Balg$ and every
$F$-stable Levi complement of $\Palg$. Because $F$ and $F'$ 
commute, $F'$ also stabilizes $\Palg^F$ and $\Lalg^F$. 
Hence, we have
$F'(R_{\Lalg}^{\Galg}(\phi))=R_{\Lalg}^{\Galg}(F'(\phi))$ for every
$\phi\in\Irr(\Lalg^F)$.
To prove that
$F'(^*R_{\Lalg}^{\Galg}(\chi))={^*\!R_{\Lalg}^{\Galg}(F'(\chi))}$ for
every $\chi\in\Irr(\Galg^F)$ it is sufficient to show that we have
$\cyc{F'(^*R_{\Lalg}^{\Galg}(\chi)),F'(\psi)}_{\Lalg^F}=
\cyc{^*R_{\Lalg}^{\Galg}(F'(\chi)),F'(\psi)}_{\Lalg^F}$ for every
$\psi\in\Irr(\Lalg^F)$. We have
\begin{eqnarray*}
{\cyc{F'(^*R_{\Lalg}^{\Galg}(\chi)),F'(\psi)}_{\Lalg^F}}&=&
{\cyc{^*R_{\Lalg}^{\Galg}(\chi),\psi}_{\Lalg^F}}\\
&=&{\cyc{\chi,R_{\Lalg}^{\Galg}(\psi)}_{\Galg^F}}\\
&=&{\cyc{F'(\chi),F'(R_{\Lalg}^{\Galg}(\psi)}_{\Galg^F}}\\
&=&{\cyc{F'(\chi),R_{\Lalg}^{\Galg}(F'(\psi))}_{\Galg^F}}\\
&=&{\cyc{^*R_{\Lalg}^{\Galg}(F'(\chi)),F'(\psi)}_{\Lalg^F}}.
\end{eqnarray*}
Using the definition of the duality functor $D_\Galg$ in
(\ref{eq:dualite}), we deduce that
$F'(D_{\Galg}(\chi))=D_{\Galg}(F'(\chi))$. Applying~$D_{\Galg}$ to
equation~(\ref{eqregulier}), we get
$$F'(\rho_{s,z})=\rho_{F'^{*-1}(s),F'(z)}.$$
This yields the claim.
\end{proof}

\subsection{Automorphisms and Jordan decomposition}
\label{rk:jordan}
In this subsection, we consider how the action of a Frobenius map $F'$ 
commuting with $F$ on the set of regular characters and the set of
semisimple characters behaves with respect to the Jordan decomposition
of characters. 

For every semisimple $s\in\Galg^{*F^*}$, let
$\mathcal{E}(\Galg^F,s)\subseteq\Irr(\Galg^F)$ be the corresponding
Lusztig series. Note that the Lusztig series give rise to a partition
of~$\Irr(\Galg^F)$; see~\cite[11.2]{BonnafeSLn}. Moreover, For every
semisimple element $s\in\Galg^{*F^*}$, there is a bijection
$$\psi_{s}:\mathcal{E}(C(s)^{F^*},1)\rightarrow\mathcal{E}(\Galg^F,s),$$
where $C(s)=\Cen_{\Galg^*}(s)$; see~\cite[13.23]{DM}. 
When the centralizer $C(s)$ is not connected,~the set
$\mathcal{E}(C(s)^{F^*},1)$ is defined as the set of constituents of
$\Ind_{C(s)^{\circ F^*}}^{C(s)^{F^*}}(\phi)$ where~$\phi$ runs through 
$\mathcal{E}(C(s)^{\circ F^*},1)$. Note that the trivial
character $1_{C(s)^{\circ}}$ and the Steinberg character
$\operatorname{St}_{C(s)^{\circ}}$ of $C(s)^{\circ F^*}$ (we identify
these two characters when~$C(s)$ is a torus) extend to
$C(s)^{F^*}$.
So,~\cite[6.17]{isaacs} implies that the extensions of these
characters are labelled by the irreducible characters of
$C(s)^{F^*}/C(s)^{\circ F^{*}}\simeq A_{\Galg^*}(s)^{F^*}$. For
$\epsilon\in\Irr( A_{\Galg^*}(s)^{F^*})$, let $1_{\epsilon}$
and $\operatorname{St}_{\epsilon}$ be the corresponding extension of
$1_{C(s)^{\circ}}$ and $\operatorname{St}_{C(s)^{\circ}}$,
respectively. Put $\cal
B_s=\{1_{\epsilon},\,\operatorname{St}_{\epsilon}\ |\
\epsilon\in\Irr(A_{\Galg^*}(s)^{F^*}\}$.

Now, we will deduce from Theorem~\ref{imageF} that 
$\psi_s$ can be chosen such that $\psi_s|_{\cal B_s}$ is compatible with the action
of a Frobenius map $F':\Galg\rightarrow\Galg$ commuting with~$F$. This
can be obtained as follows. Let $i:\Galg\rightarrow\widetilde{\Galg}$ be the embedding
as above. Put $\ker'(i^*)=\ker(i^*)\cap
[\widetilde{\Galg}^*,\widetilde{\Galg}^*]$ and let $\cal Z(\Galg)$ be the
component group of the center $\Zalg$ of $\Galg$. Recall that there is an
isomorphism $\hat{\omega}:\cal Z(\Galg)\rightarrow\Irr(\ker'(i^*))$
which induces an isomorphism $\hat{\omega}^0:H^1(F,\Zalg)\rightarrow
\Irr(\ker'(i^*)^{F^{*}})$;
see~\cite[(4.11)]{BonnafeSLn}. Moreover, for every semisimple
$s\in\cal \Galg^{*F^*}$, we define an injective homomorphism
\begin{equation}
\label{phis}
\varphi_s:A_{\Galg^*}(s)\rightarrow
\ker'(i^*),\,g\mapsto
\widetilde{g}\,\widetilde{s}\,\widetilde{g}^{-1}\,\widetilde{s}^{-1},
\end{equation}
where $\widetilde{g}$ and $\widetilde{s}$ are elements of
$\widetilde{\Galg}$ such that $i^*(\widetilde{g})=g$ and
$i^*(\widetilde{s})=s$. This morphism induces a surjective morphism
$\hat{\varphi}_s:\Irr(\ker'(i^*))\rightarrow \Irr(A_{\Galg^*}(s))$.
Note that $\hat{\varphi}_s$ induces a surjective morphism from
$\Irr(\ker'(i^*)^{F^*})$ to $\Irr(A_{\Galg^*}(s)^{F^*})$.
Composing this last morphism with $\hat{\omega}^0$, we obtain a
surjective map  $ \hat{\omega}_s^0:H^1(F,\Zalg)\rightarrow
\Irr(A_{\Galg^*}(s)^{F^*}) $; see~\cite[8.A]{BonnafeSLn} for more
details.
The Frobenius map $F'$ induces an automorphism on $H^1(F,\Zalg)$, because $F$
and $F'$ commute.  Moreover, $F'^*$ induces an isomorphism from
$A_{\Galg^*}(F'^{*-1}(s))$ to $A_{\Galg^*}(s)$. Thus, by dualizing, we
obtain the following isomorphism
$$F'^*:\Irr\left(A_{\Galg^*}(s)\right) \rightarrow 
\Irr\left(A_{\Galg^*}(F'^{*-1}(s))\right),\,
\phi\mapsto\phi\circ F'^*.$$
Now, consider the diagram:
\begin{equation}
\xymatrix{
   H^1(F,\Zalg) \ar[rrrr]^{\hat{\omega}_s^0} \ar[d]^{F'} &&&&
   \Irr(A_{\Galg^*}(s)^{F^*})\ar[d]^{F'^*}\\
   H^1(F,\Zalg) \ar[rrrr]^ {\hat{\omega}_{F'^{*-1}(s)}^0}&&&&
   \Irr(A_{\Galg^*}(F'^{*-1}(s))^{F^*})  
   } 
\label{diag}
\end{equation}
Fix $z\in H^1(F,\Zalg)$ and $g\in A_{\Galg^*}(F'^{*-1}(s))^{F^*}$.
Equation~(\ref{phis}) implies
$$\varphi_s(F'^*(g))=F'^*\left(\varphi_{F'^{*-1}(s)}(g)\right).$$
Then one has
\begin{eqnarray*}
F'^*\left(\hat{\omega}^0_s(z)\right)(g)&=&\hat{\omega}_s^0(z)(F'^*(g))\\
&=&\hat{\omega}^0(z)\circ\varphi_s(F'^{*}(g))\\
&=&\hat{\omega}^0(z)\left(F'^*\left(\varphi_{F'^{*-1}(s)}(g)\right)\right)\\
&=&\omega^0\left(F'^*(\varphi_{F'^{*-1}(s)}(g))\right)(z)\\
&=&\omega^0\left(\varphi_{F'^{*-1}(s)}(g)\right)(F'(z))
\end{eqnarray*}
Here, $\omega^0$ is defined as in \cite[4.11]{BonnafeSLn}, and 
the last equality comes from~\cite[4.10]{BonnafeSLn}. It follows that 
$$F'^*\left(\hat{\omega}_s^0(z)\right)=\hat{\omega}_{F'^{*-1}(s)}^0\left(F'(z)\right).$$
Hence, diagram~(\ref{diag}) is commutative.
Now, we define $\psi_s$ on $\cal B_s$ by setting
$$\psi_{s}(1_{\hat{\omega}_{s}^0(z)})=
\rho_{s,z}\quad\textrm{and}\quad
\psi_{s}(\operatorname{St}_{\hat{\omega}_{s}^0(z)})=
\chi_{s,z}.$$
Note that, by~\cite[11.13]{BonnafeSLn}, we have
$\rho_{s,z}=\rho_{s,z'}$ (resp. $\chi_{s,z}=\chi_{s,z'}$) if and only
if $z'z^{-1}\in\ker(\hat{\omega}_{s}^0)$. Thus
$\psi_{s}|_{\cal B_{s}}$ is well-defined. Put
$$\Psi:\bigcup_{s}\cal
B_s\rightarrow\Irr(\Galg^F),\,1_{\hat{\omega}_{s}^0(z)}\mapsto
\rho_{s,z},\,\operatorname{St}_{\hat{\omega}_{s}^0(z)}\mapsto\chi_{s,z},$$
where $s$ runs through a set of representatives for the geometric
conjugacy classes of semisimple elements of $\Galg^{*F^*}$.
The commutativity of diagram~(\ref{diag}) and 
Theorem~\ref{imageF} imply that $\Psi$ is $F'$-equivariant, as
required.


\section{Characters of Borel subgroups}\label{partborel}

We continue to use the setup from Section~\ref{nota}. Additionally, we
assume that $\Galg$ is simple and that $\widetilde{\Galg}^F$ is not
one of the groups $B_m(2), C_m(2), G_2(2), G_2(3), F_4(2)$.
Let $F':\Galg\rightarrow\Galg$ be a Frobenius map of $\Galg$ commuting
with $F$ such that $\Talg$ and~$\Balg$ are $F'$-stable as in
Section~\ref{partie1}. 
In this section, we consider the action of $\widetilde{\Talg}^F$ and
the Frobenius morphism~$F'$ on the set of irreducible $p'$-characters
of the Borel subgroup~$\Balg^F$.  

\subsection{$p'$-characters of the Borel subgroup and automorphisms}

Let $\Ualg_0$ be the subgroup defined in equation (\ref{eq:etoile}) and
$\Ualg_1 = \Ualg/\Ualg_0$, see also equation (\ref{eq:etoilequo}). 
Set $$\Balg_0=\Ualg_1\rtimes\Talg.$$ The group 
$\Balg_0$ is an algebraic group with a rational structure over $\F_q$
given by the Frobenius map $F$. Since $\widetilde{\Galg}^F$ is not one
of the groups listed at the beginning of this section,~\cite[Lemma
7]{howlett} implies that $\Ualg_0^F$ is the derived subgroup
of~$\Ualg^F$, so the sets $\Irr_{p'}(\Balg^F)$ and $\Irr(\Balg_0^F)$ are
$\mathcal{A}$-equivalent, where $\mathcal{A}$ denotes the set of 
automorphisms of $\Balg_0^F$ leaving $\Ualg_1^F$ and $\Talg^F$
invariant. 

Let $\widetilde{\Omega}$ and $\Omega$ be the sets of
$\widetilde{\Talg}^F$-orbits and $\Talg^F$-orbits on
$\Irr(\Ualg_1^F)$, respectively.
For $J = \{\alpha_{j_1}, \alpha_{j_2}, \dots, \alpha_{j_l}\} \subseteq
\Delta$, set
\begin{eqnarray*}
\Ualg_J & := & \{ x_{\alpha_{j_1}}(u_1) x_{\alpha_{j_2}}(u_2) \cdots
x_{\alpha_{j_l}}(u_l) \, | \, u_1, \dots,
u_l \in \overline{\F}_q \} \text{   and}\\
\Ualg_J^* & := & \{ x_{\alpha_{j_1}}(u_1) x_{\alpha_{j_2}}(u_2) \cdots
x_{\alpha_{j_l}}(u_l) \, | \, u_1, \dots,
u_l \in \overline{\F}_q^\times \}.
\end{eqnarray*}
Let $\Irr^*(\Ualg_J^F)$ be the set of all $\chi \in \Irr(\Ualg_J^F)$
such that $\Res^{\Ualg_J^F}_{\Xalg_\alpha^F}(\chi)$ is non trivial for
all $\alpha \in J$. By extending every $\phi \in \Irr(\Ualg_J^F)$
trivially, we can identify $\Irr^*(\Ualg_J^F)$ with a subset of
$\Irr(\Ualg_1^F)$ in a natural way. With this identification, we have
\begin{equation}\label{eq:decorb}
\Irr(\Ualg_1^F) = \bigsqcup_{J \subseteq \Delta} \Irr^*(\Ualg_J^F),
\end{equation}
and each $\Irr^*(\Ualg_J^F)$ is invariant under the action of
$\widetilde{\Talg}^F$, $\Talg^F$ and $F$.
In the proof of \cite[5.1]{Br8} it is shown that the irreducible
characters of $\Balg_0^F$ (or equivalently, the $p'$-characters of
$\Balg^F$) can be labelled as follows. We can parametrize the elements
of $\widetilde{\Omega}$ by the subsets of $\Delta$. More precisely, 
for $J\subseteq \Delta$, the corresponding
$\widetilde{\Talg}^F$-orbit is $\Irr^*(\Ualg_J^F)$. Fix
$\phi_J\in\Irr^*(\Ualg_J^F)$ and choose $t_j\in\widetilde{\Talg}^F$
such that the characters $\phi_{J,j}={}^{t_j}\phi_J$ form a set of
representatives for the $\Talg^F$-orbits of $\Irr^*(\Ualg_J^F)$, where
we choose $t_1=1$. Then the $\Talg^F$-orbit of $\phi_{J,j}$ will be
denoted by $\omega_{J,j}$. Since $\omega_{J,j}$ and $\omega_{J,1}$ are
conjugate by an element of $\widetilde{\Talg}^F$, the size
$|\omega_{J,j}|$ does not depend on $j$. Furthermore
\begin{equation}\label{eq:unionorb}
\Irr^*(\Ualg_J^F)=\bigsqcup_{j=1}^{i(J)}\omega_{J,j}\quad \textrm{where
}i(J)=|\Irr^*(\Ualg_J^F)|/|\omega_{J,1}|.
\end{equation}
\begin{remark}\label{rk:choix}
Note that if $\Lalg_J$ denotes the standard Levi subgroup corresponding to $J$,
the proof of~\cite[5.2]{Br8} shows that there is an element 
 $t_{z_j}\in\Talg$ for some $z_j\in H^1(F,Z(\Lalg_J))$ satisfying
$t_{z_j}^{-1}F(t_{z_j})\in z_j$ and
\begin{equation}
\label{eq:defph}
\phi_{J,j}=^{t_j}\!\phi_J=^{t_{z_j}}\!\phi_J.
\end{equation}
%
\end{remark}
Let $I_{\phi_{J,j}}=\Ualg^F\rtimes \Cen_{\Talg^F}(\phi_{J,j})$ be
the inertia subgroup of $\phi_{J,j}$ in $\Balg_0^F$. 
Note that~$\phi_{J,j}$ extends to $I_{\phi_{J,j}}$ by setting, for
$u\in\Ualg^F$ and $t\in\Cen_{\Talg^F}(\phi_{J,j})$,
\begin{equation}\label{eq:hat}
\hat{\phi}_{J,j}(ut)=\phi_{J,j}(u).
\end{equation}
Thus, the irreducible characters of $\Irr(\Balg_0^F)$ are the characters
$\Ind_{I_{\phi_{J,j}}}^{\Balg_0^F}(\hat{\phi}_{J,j}\otimes\psi)$
for $J$ and $j$ as above and $\psi\in\Irr(\Cen_{\Talg^F }(\phi_{J,j}))$.

We now will describe more precisely the group $\Cen_{\Talg^F
}(\phi_{J,j})$.
As usual, we write $(X(\Talg),Y(\Talg), \Phi, \Phi^\vee)$ for the root 
datum corresponding to
$(\Galg, \Talg)$ in the sense of \cite[Theorem 3.17]{DM}. In
particular, $X(\Talg)$ is the character group and $Y(\Talg)$ the
cocharacter group of $\Talg$. Choose $\Z$-bases
$b = \{b_1, \dots, b_r\}$ and $b' = \{b_1', \dots, b_r'\}$ of
$X(\Talg)$ and $Y(\Talg)$ respectively, such that $b$ and $b'$ are
dual to each other with respect to the natural pairing,
see \cite[Section~1.9]{Carter2}. By \cite[Proposition~3.1.2]{Carter2},
we have $\Talg \simeq Y \otimes_{\Z} \overline{\F}_q^\times$ as abelian
groups. Every element of $Y \otimes_{\Z} \overline{\F}_q^\times$ can be
written uniquely as $\sum_{i=1}^r b_i' \otimes t_i$ with
$t_i \in \overline{\F}_q^\times$ and we write $(t_1, t_2, \dots, t_r)$
for the corresponding element of $\Talg$. Note that
$|b|=|b'|=|\Delta|$ because $\Galg$ is simple. 

Since $\Phi \subset X(\Talg)$, we can write every element of
$\Phi$ as a $\Z$-linear combination of~$b$ and can define a matrix
$A = (a_{ij}) \in \Z^{r \times r}$ as follows: Let the $i$th row of
$A$ consist of the coefficients of the simple root $\alpha_i$ written
as a linear combination of~$b$. For a simple root
$\alpha_i \in \Delta$, the action of a torus element 
$t = (t_1, t_2, \dots, t_r) \in \Talg$ on $\Xalg_{\alpha_i}$ is given by  
\begin{equation} \label{eq:toronuni}
^tx_{\alpha_i}(u)=x_{\alpha_i}\left(u \cdot \prod_{j=1}^r t_j^{a_{ij}}\right).
\end{equation}
Fix $J = \{\alpha_{j_1}, \alpha_{j_2}, \dots, \alpha_{j_l}\}\subseteq \Delta$
as above. For $x \in \Ualg_J^*$, we have 
\begin{equation}\label{eq:Tj}
\Talg_J := \Stab_{\Talg} (x) = 
\left\{t\in\Talg\,|\,\prod_{j=1}^r t_j^{a_{j_k,j}}=1 \text{   for
} k = 1, \dots, l \right\}.
\end{equation}
In particular, the stabilizers in $\Talg$ of all $x \in \Ualg_J^*$ 
coincide. Note that the stabilizer $\Talg_J$ is an $F$-stable diagonalizable group.
According to~\cite[13.2.5(1)]{Springer}, $\Talg_J^\circ$ is a split subtorus of $\Talg$ and
$\Talg_J=\Talg_J^\circ\times H_J$, where $H_J$ is a finite group isomorphic
to the torsion group of $X(\Talg_J)$.
Note that $\Talg_J^{\circ}$ and $H_J$ are $F$-stable and
$\Talg_J^F=\Talg_J^{\circ F}\times H_J^F$.

\begin{lemma}
With the above notation, for $J\subseteq \Delta$, the group $\Talg_J^F$
is the centralizer in $\Talg^F$ of all irreducible characters in
$\Irr^*(\Ualg_J^F)$. Write $\phi_{J,j}$ for the character of $\Irr^*(\Ualg_J^F)$
in the $\Talg^F$-orbit $\omega_{J,j}$ constructed from $\phi_J$ as in equation (\ref{eq:defph}),
and let $i(J)$ be as in
equation~(\ref{eq:unionorb}). Set
$$\chi_{J,j,\psi}=\Ind_{\Ualg^F\rtimes
\Talg_J^F}^{\Balg_0^F}(\hat{\phi}_{J,j}\otimes \psi),$$
where $\hat{\phi}_{J,j}$ is the extension of $\phi_{J,j}$ defined in
equation~(\ref{eq:hat}). Then, one has
$$\Irr(\Balg_0^F)=\left\{\chi_{J,j,\psi}\,|\,J\subseteq \Delta,\,1\leq j\leq
i(J),\,\psi\in\Irr(\Talg_J^F)\right\}.$$
\label{charB}
\end{lemma}
\begin{proof}
This is just Clifford theory.
\end{proof}

\begin{remark}\label{rk:centrelevi}
Note that $\Talg_J$ is the center of the Levi subgroup $\Lalg_J$. In
particular, one has $H_J\simeq\cal Z(\Lalg_J)$, where $\cal
Z(\Lalg_J)=Z(\Lalg_J)/Z(\Lalg_J)^{\circ}$.
\end{remark}

\begin{remark}\label{rk:charBFrob}
We will make some choices in the labelling of the irreducible
characters of $\Balg_0^{F}$ given in Lemma~\ref{charB}. In the
following, if $J\subseteq\Delta$ is $F'$-stable (implying that $F'$
acts on $\Irr^*(\Ualg_J^F)$), the character $\phi_J$ of
$\Irr^*(\Ualg_J^F)$ used for the parametrization of $\omega_{J,j}$ in
Lemma~\ref{charB} will be chosen $F'$-stable, which is possible by
Lemma~\ref{unistable}.  
\end{remark}

\begin{lemma}\label{deforb}
For $J\subseteq \Delta$, $\psi\in\Irr(\Talg_J^F)$ and $i = i(J)$ as in
equation~(\ref{eq:unionorb}), the set
$$D_{J,\psi,i}:=\{\chi_{J,j,\psi}\in\Irr(\Balg_0^{F})\,|\,1\leq j\leq
i\}$$
is a $\widetilde{\Talg}^F$-orbit of $\Irr(\Balg_0^F)$ of size $i$
and all $\widetilde{\Talg}^F$-orbits of $\Irr(\Balg_0^F)$ of size
$i$ arise in this way.
\end{lemma}

\begin{proof}
Fix $t\in\widetilde{\Talg}^F$ and let $j$ be the integer such that
${}^t\phi_J$ and $\phi_{J,j}$ are $\Talg^F$-conjugate. Note that
$\Ualg^F\rtimes\Talg_J^F$ is $\widetilde{\Talg}^F$-invariant, implying
that for $\psi\in\Irr(\Talg_J^F)$
$${^t}\chi_{J,1,\psi}=\Ind_{\Ualg^F\rtimes\Talg_J^F}^{\Balg_0^F}\left({}^t
\hat{\phi}_J\otimes{}^t\psi\right).$$
Furthermore, ${}^t\hat{\phi}_J=\widehat{ {}^t\phi}_J$ and $\Talg_J^F$ is the
inertia subgroup of ${}^t{\phi}_J$ in $\Talg^F$ (because $\Talg_J^F$ is
the inertia subgroup of any character of $\Ualg_J^F$). Moreover,
$\Talg_J^F$ and $\widetilde{\Talg}^F$ commute, so ${}^t\psi=\psi$ and
then $${}^t\chi_{J,1,\psi}=\chi_{J,j,\psi}.$$
Conversely, ${}^{t_j}\chi_{J,1,\psi}=\chi_{J,j,\psi}$ and the result follows.
\end{proof}

The following lemma describes the action of the Frobenius morphism
$F'$ on the set of $\widetilde{\Talg}^F$-orbits on $\Irr(\Balg_0^F)$.

\begin{lemma}We assume that the convention of
Remark~\ref{rk:charBFrob} holds.
Let $F':\Galg\rightarrow\Galg$ be a Frobenius map commuting with $F$ 
such that $\Talg$ and $\Balg$ are $F'$-stable.
For an $F'$-stable $J\subseteq\Delta$, an integer $i = i(J)$ as in 
equation~(\ref{eq:unionorb}) and $\psi\in\Irr(\Talg_J^F)$, 
one has
$$F'(D_{J,\psi,i})=D_{J,F'(\psi),i}.$$
More precisely, for every $1\leq j\leq i$, we have
$$F'(\chi_{J,j,\psi})=\chi_{J,j',F'(\psi)},$$
where $j'$ is such that $z_{j'}=F'(z_j)$ for $z_j\in
H^1(F',Z(\Lalg_{J}))$ as in Remark~\ref{rk:choix}.
\label{la:FrobchrB}
\end{lemma}

\begin{proof}Note first that $F'$ permutes the
$\widetilde{\Talg}^F$-orbits of $\Irr(\Balg_0^F)$ because
$F'(\widetilde{\Talg}^F)=\widetilde{\Talg}^F$.
Let $\chi_{J,j,\psi}\in D_{J,\psi,i}$.
Then $F'$ acts on $\Irr(\Ualg_J^F)$ and fixes $\phi_J$ (see
Remark~\ref{rk:charBFrob}). In particular, $\Talg_J^F$ is $F'$-stable.
Let $t_{z_j}\in\Talg$ be such that $\phi_{J,j}={^{t_{z_j}}\phi_J}$ and
$z_j\in H^1(F,Z(\Lalg_J))$. Since $\phi_J$ is $F'$-stable, one has
$$F'(\phi_{J,j})={^{F'(t_{z_j})}\phi_J}={^{t_{F'(z_j)}}\phi_J}=\phi_{J,j'},$$
where $j'$ is such that $F'(z_j)=z_{j'}$.
It follows
$$F'(\chi_{J,j,\psi})=\Ind_{\Ualg^F\Talg_J^F}^{\Balg^F}
\left(\hat{\phi}_{J,j'}\otimes F'(\psi)\right)=\chi_{J,j',F'(\psi)},$$
because $F'(\hat{\phi}_{J,j})=\hat{\phi}_{J,j'}$
The result follows.
\end{proof}

\subsection{Characters of the Borel subgroup and central characters}
\begin{lemma} \label{la:IrrPprimeBnu}
Let $\nu \in \Irr(Z(\Galg^F))$ and write $d := |Z(\Galg^F)|$. Then
\[
|\Irr_{p'}(\Balg^F | \nu) | = \frac{1}{d} 
|\Irr_{p'}(\Balg^F)|.
\]
\end{lemma}

\begin{proof}
Write $\Zalg := Z(\Galg)$. With this notation, we have
\[
\Irr_{p'}(\Balg^F | \nu) = \{ \chi \in \Irr_{p'}(\Balg^F) \,
| \, \Res_{\Zalg^F}^{\Balg^F}(\chi) = \chi(1) \cdot \nu\}
\]
and $\Irr_{p'}(\Balg^F) = \sqcup_{\nu \in \Irr(\Zalg^F)} \Irr_{p'}(\Balg^F | \nu)$.
Let $\nu, \nu' \in \Irr(\Zalg^F)$. We can extend $\nu, \nu'$ to linear 
characters of $\Balg^F$ and denote these extensions also by
$\nu$ and $\nu'$, respectively. The map $\Irr_{p'}(\Balg^F
| \nu) \rightarrow \Irr_{p'}(\Balg^F | \nu')$,
$\chi \mapsto \chi \cdot \overline{\nu} \cdot \nu'$ is a bijection,
where $\overline{\nu}$ is the complex-conjugate of $\nu$.
So, $|\Irr_{p'}(\Balg^F)| = |\Zalg^F| \cdot |\Irr_{p'}(\Balg^F
| \nu)|$ and the claim follows.
\end{proof}

\begin{lemma}Let $J$ be a subset of $\Delta$, $\psi\in\Irr(\Talg_J^F)$
and $\chi_{J,j,\psi}\in\Irr(\Balg_0^F)$ as Lemma~\ref{charB}.
For $\nu\in\Irr(\Zalg^{F})$, one has
$$\cyc{\chi_{J,j,\psi},\Ind_{\Zalg^{F}}^{\Balg_0^{F}}(\nu)}_{\Balg_0^{F}}\neq
0 \Longleftrightarrow  \cyc{\Res_{\Zalg^{F}}^{\Talg_J^F}(\psi),\nu}_{\Zalg^{F}}\neq
0.$$
\end{lemma}
\begin{proof}
By the definition of induced characters \cite[(5.1)]{isaacs}, we have
$$\Res_{\Zalg^{F}}^{\Balg_0^{F}}(\chi_{J,j,\psi})
= \chi_{J,j,\psi}(1) \cdot \Res_{\Zalg^{F}}^{\Talg_J^F}(\psi)$$ and the
claim follows from Frobenius reciprocity.
\end{proof}


\section{Equivariant bijections}\label{partie2}

We use the notation and setup from the previous sections. In
particular, we have an embedding
$i:\Galg\rightarrow\widetilde{\Galg}$ where $\widetilde{\Galg}$ is a
connected reductive group defined over $\F_q$ with connected center
and Frobenius map $F$, such that $i(\Galg)$ is the derived subgroup
of $\widetilde{\Galg}$. 
Let $\widetilde{\Talg}$ be the $F$-stable maximal torus of 
$\widetilde{\Galg}$ containing $\Talg$
and $\widetilde{\Balg}$ a Borel subgroup of $\widetilde{\Galg}$
containing $\Talg$ and $\Balg$. 
The quotient $\widetilde{\Talg}^F/\Talg^F=\Talg'^F$ acts
on~$\Galg^F$. Let $D$ be the subgroup of $\Out(\Galg^F)$ induced by
this action. We will assume throughout this whole section that the
following hypothesis is satisfied. 

\begin{hypothese}\label{hypoG}
Let $\Galg$ be as above. Assume that
\begin{itemize}
\item The algebraic group $\Galg$ is simple.
\item The prime $p$ is nonsingular for $\Galg$.
\item The automorphism induced by $F$ on $W$ is trivial.
\item The order $d := |D|$ is a prime number.
\end{itemize}
\end{hypothese}

\begin{remark}Note that if Hypothesis \ref{hypoG} holds, then $d\neq p$.
\end{remark}

In this section, we are going to construct bijections between
the sets $\Irr_{p'}(\Galg^F|\nu)$ and $\Irr_{p'}(\Balg^F|\nu)$ for
fixed $\nu\in\Irr(\Zalg^F)$, which are compatible with the action of
certain groups of automorphisms.

Recall that $\Galg$ is generated by the elements $x_{\alpha}(t)$
where $\alpha\in\Phi$ and $t\in\overline{\F}_p$ (because $\Galg$ is
simple). Moreover, since $F$ acts trivially on $W$, we can choose $x_{\alpha}$
such that  
$$F(x_{\alpha}(t))=
x_{\alpha}(t^q)\quad\textrm{for all
}\alpha\in\Phi\ \textrm{and }t\in\overline{\F}_p.$$
We then define a bijective algebraic group homomorphism 
$F_0:\Galg\rightarrow\Galg$ satisfying
$F_0(x_{\alpha}(t))=x_{\alpha}(t^p)$ for all
$\alpha\in\Phi$ and $t\in\overline{\F}_p$.
Note that the map $F_0$ defines an $\F_p$-rational structure on $\Galg$.
Moreover, if $q=p^n$ for a positive integer $n$, then $F_0^n=F$. 

\subsection{Automorphisms of $\Galg^F$} \label{auto}

Since $F_0$ and $F$ commute, we have $F_0(\Galg^F)=\Galg^F$.
Thus, $F_0$ induces an automorphism of $\Galg^F$, also denoted by
$F_0$ in the following. We set $K=\cyc{F_0}\subseteq\Out(\Galg^F)$.
Note that the finite group $K$ has order $n$ and is the group of field
automorphisms of~$\Galg^F$, see~\cite[12.2]{Carter1}.   

We write $A\subseteq\Out(\Galg^F)$ for the subgroup of $\Out(\Galg^F)$
generated by $K$ and $D$.
Note that, by construction, the groups $\Talg^F$, $\widetilde{\Talg}^F$,
$\Balg^F$, $\widetilde{\Balg}^F$, $\Ualg^F$, $\Ualg_1^F$ are
$K$-stable, $D$-stable and then $A$-stable.
The group $A$ acts on $\Irr_{p'}(\Galg^F)$ and
$\Irr_{p'}(\Balg^F)$. For every subgroup $H\subseteq\Out(\Galg^F)$, we
denote by $\mathcal{O}_H$ and $\mathcal{O}_H'$ the set of $H$-orbits on
$\Irr_{p'}(\Galg^F)$ and $\Irr_{p'}(\Balg^F)$, respectively.

In this whole Section~\ref{partie2}, let $\nu$ denote a linear
character of $\Zalg^F$ and let $A_{\nu}$ be the subgroup of $A$ fixing
$\nu$. Note that $D$ acts trivially on $\Zalg^F$. Then 
\begin{equation}\label{eq:Anu}
A_{\nu}=D\rtimes K_{\nu}, 
\end{equation}
where $K_\nu$ is the subgroup of field automorphisms fixing $\nu$.
Then the group $A_{\nu}$ acts on 
$\Irr_{p'}(\Galg^F|\nu)$ and $\Irr_{p'}(\Balg^F|\nu)$.
For every subgroup $H$ of $A_\nu$, let $\mathcal{O}_{H,\nu}$ and
$\mathcal{O}_{H,\nu}'$ be the set of $H$-orbits on
$\Irr_{p'}(\Galg^F|\nu)$ and $\Irr_{p'}(\Balg^F|\nu)$, respectively.  

\begin{lemma}\label{KD}
The group $K$ acts on the sets $\mathcal{O}_D$ and
$\mathcal{O}_D'$. And the group $K_\nu$ acts on the sets
$\mathcal{O}_{D,\nu}$ and $\mathcal{O}_{D,\nu}'$. 
\end{lemma}

\begin{proof}
The first statement follows from $D \unlhd A$,
see \cite[2.5.14]{sol}. The second statement is then clear.
\end{proof}

\subsection{A $D$-equivariant bijection respecting central characters} 

The following lemma describes a sufficient condition for the existence 
of a $D$-equivariant bijection between $\Irr_{p'}(\Galg^F|\nu)$ and
$\Irr_{p'}(\Balg^F|\nu)$. In the subsequent remark we then show that
this condition is satisfied if the relative McKay conjecture is true
for $\widetilde{\Galg}^F$ and $\Galg^F$ at the prime $p$.

\begin{lemma}\label{la:diagequivcentral}
Suppose that Hypothesis~\ref{hypoG} is
satisfied and assume that
$$|\Irr_{p'}(\widetilde{\Galg}^F|\nu)|=|\Irr_{p'}(\widetilde{\Balg}^F|\nu)| 
\text{   and   } |\Irr_{p'}(\Galg^F|\nu)|=|\Irr_{p'}(\Balg^F|\nu)|.$$ 
Then there is a $D$-equivariant
bijection between $\Irr_{p'}(\Galg^F|\nu)$ and $\Irr_{p'}(\Balg^F|\nu)$.
\end{lemma}

\begin{proof}
Fix $\nu\in\Irr(\Zalg^F)$. Since $|D| = d$ is prime, the $D$-orbits on
$\Irr_{p'}(\Galg^F|\nu)$ and $\Irr_{p'}(\Balg^F|\nu)$ have size $1$ or $d$. 
For $i\in\{1,d\}$, let $N_i(\nu)$ (resp. $N'_i(\nu)$) be the set of
$D$-orbits on $\Irr_{p'}(\Galg^F|\nu)$ (resp. $\Irr_{p'}(\Balg^F|\nu)$) of size $i$.  
Hence, we have
$$|\Irr_{p'}(\Galg^F|\nu)|=|N_1(\nu)|+d|N_d(\nu)|\quad\textrm{and}\quad
|\Irr_{p'}(\Balg^F|\nu)|=|N'_1(\nu)|+d|N'_d(\nu)|.$$
Let $\omega\in\mathcal{O}_{D,\nu}$. Then there is a semisimple element
$s$ of $\Galg^{*F^*}$, such that $\omega$ is the set of the
constituents of $\rho_s$. Let $\widetilde{s}$ be a semisimple element
of $\widetilde{\Galg}^{*F^*}$ satisfying $i^*(\widetilde{s})=s$.
Then $\rho_{\widetilde{s}}\in\Irr_{p'}(\widetilde{\Galg}^F|\nu)$ and 
every $\widetilde{s}'$ restricting to $\rho_s$ also lies in
$\Irr_{p'}(\widetilde{\Galg}^F|\nu)$. Let $\chi \in \omega$ be an
irreducible constituent of $\rho_s$. Since $\Talg'^F$ is
abelian of order $q-1$ and~$\rho_s$ is multiplicity free, Clifford
theory~\cite[Problem (6.2)]{isaacs} implies
$$|\Irr(\widetilde{\Galg}^F|\rho_s)|=\frac{q-1}{|A_{\Galg^*}(s)^{F^*}|}.$$

Similarly, the results in Section~\ref{partborel} on the action of
$\widetilde{\Talg}^F$ and $\Talg^F$ on  
$\Irr(\Ualg_1^F)$ imply that, for every $\omega\in\mathcal{O}'_{D,\nu}$,
there are exactly $(q-1)/|\omega|$ irreducible characters 
$\widetilde{\chi} \in \Irr_{p'}(\widetilde{\Balg}|\nu)$ such that
$\omega$ is the set of constituents of 
$\Res_{\Balg^F}^{\widetilde{\Balg}^F}(\widetilde{\chi})$. 
So,
\begin{eqnarray*}
|\Irr_{p'}(\widetilde{\Galg}^F|\nu)| & = &
 (q-1)|N_1(\nu)|+\frac{q-1}{d}|N_d(\nu)|\quad\textrm{and}\\
|\Irr_{p'}(\widetilde{\Balg}^F|\nu)| & = &
 (q-1)|N'_1(\nu)|+\frac{q-1}{d}|N'_d(\nu)|.
\end{eqnarray*}
By assumption,
$|\Irr_{p'}(\widetilde{\Galg}^F|\nu)|=|\Irr_{p'}(\widetilde{\Balg}^F|\nu)|$
and $|\Irr_{p'}(\Galg^F|\nu)|=|\Irr_{p'}(\Balg^F|\nu)|$. So, we can deduce that
$$\left\{
\begin{array}{lcl}
(|N_1(\nu)|-|N'_1(\nu)|)+d(|N_d(\nu)|-|N'_d(\nu)|)&=&0,\\
d(|N_1(\nu)|-|N'_1(\nu)|)+(|N_d(\nu)|-|N'_d(\nu)|)&=&0.\\
\end{array}\right.$$
We can conclude $|N_1(\nu)|=|N'_1(\nu)|$ and $|N_d(\nu)|=|N'_d(\nu)|$. 
Thus, there is a
$D$-equivariant bijection between $\Irr_{p'}(\Galg^F|\nu)$ and
$\Irr_{p'}(\Balg^F|\nu)$ which can be described as follows:
First, for $i\in\{1,d\}$, we choose any bijection
$f_i:N_i(\nu)\rightarrow N'_i(\nu)$. For $\omega\in
N_i(\nu)$, we choose any $x_{\omega}\in\omega$ and any $y_{\omega}\in
f_i(\omega)$.
We set $D=\cyc{\delta}$.
We define $f:\Irr_{p'}(\Galg^F|\nu)\rightarrow\Irr_{p'}(\Balg^F|\nu)$ by
$$f(\delta(x_{\omega}))=\delta(y_{\omega})\quad\textrm{for }\omega\in
\mathcal{O}_{D,\nu}.$$ So, by construction, the map $f$ is a $D$-equivariant
bijection.
\end{proof}

\begin{remark}
Suppose $\widetilde{\Galg}^F$ satisfies the relative McKay
conjecture at the prime $p$, that is 
$|\Irr_{p'}(\widetilde{\Galg}^F|\widetilde{\nu})|=|\Irr_{p'}(\widetilde{\Balg}^F|\widetilde{\nu})|$
holds for all
$\widetilde{\nu} \in \Irr(Z(\widetilde{\Galg}^F))$. Then, for every
$\nu\in\Irr(\Zalg^F)$, one has
$$|\Irr_{p'}(\widetilde{\Galg}^F|\nu)|=|\Irr_{p'}(\widetilde{\Balg}^F|\nu)|.$$
\end{remark}
\begin{proof}
If
$\Ind_{\Zalg^F}^{Z(\widetilde{\Galg}^F)}(\nu)=\sum_{k=1}^r\widetilde{\nu_i}$, then 
\begin{eqnarray*}
|\Irr_{p'}(\widetilde{\Galg}^F|\nu)|&=&\sum_{k=1}^r
|\Irr_{p'}(\widetilde{\Galg}^F|\widetilde{\nu}_k)|\\
&=&
\sum_{k=1}^r
|\Irr_{p'}(\widetilde{\Balg}^F|\widetilde{\nu}_k)|\\
&=&
|\Irr_{p'}(\widetilde{\Balg}^F|\nu)|.
\end{eqnarray*}
\end{proof}

\subsection{Central characters and automorphisms of $\Galg^F$}

In this subsection, we study the action of the field automorphisms on
the set of $D$-orbits on $\Irr_{p'}(\Galg^F|\nu)$. In fact, we
consider the action of $F$ on the set of $D$-orbits on
$\Irr_{p'}(\Galg^{F^m}|\mu)$ for positive integers $m$ and
$\mu \in \Irr(\Zalg^{F^m})$. 
The following remark will be used to parametrize $F$-stable
semisimple conjugacy classes of $\Galg^{*F^{*m}}$.

\begin{remark}\label{rk:paramcl}
Suppose that $F$ and $F'$ are Frobenius maps on $\Galg$ which commute. 
Let $s$ be an $F^*$- and $F'^*$-stable element of $\Galg^*$. For $\alpha\in
H^1(F^*,A_{\Galg^*}(s))$, we choose $g_{\alpha}\in\Galg^*$ such that the
class of $g_{\alpha}^{-1}F^*(g_{\alpha})$ in $H^1(F^*,A_{\Galg^*}(s))$
equals $\alpha$. Let $s_{\alpha}:=g_{\alpha}sg_{\alpha}^{-1}$. Then
$s_{\alpha}\in\Galg^{*F^*}$ and
$$[s]_{\Galg^*}\cap\Galg^{*F^*}=\bigsqcup_{\alpha\in
H^1(F^*,A_{\Galg^*}(s))}[s_{\alpha}]_{\Galg^{*F^*}}.$$
Moreover, since $s$ is $F'^*$-stable,
one has $F'^*(s_{\alpha})=F'^*(g_{\alpha})sF'^*(g_{\alpha})^{-1}$,
implying that $F'^*(\alpha)$ is equal to the class of 
$F'^*(g_{\alpha}^{-1}F^*(g_{\alpha}))=F'^*(g_{\alpha}^{-1})
F^*\left(F'^*(g_{\alpha})\right)$ in $H^1(F^*,A_{\Galg^*}(s))$,
because $F^*$ and $F'^*$ commute. We then deduce that $F'^*(s_{\alpha})$
and $s_{F'^*(\alpha)}$ are $\Galg^{*F^*}$-conjugate.
\end{remark}


\begin{lemma}Suppose that $\Galg$ is simple and $d=|D|$ a prime number.
Let $s \in \Galg^{*F^*}$ be a semisimple element such that
$|A_{\Galg^*}(s)^{F^*}|=d$ and $\{s_{\alpha} \, | \, 
\alpha\in H^1(F^*,A_{\Galg^*}(s))\}$ be a set of representatives for
the $F$-rational semisimple classes corresponding to
$s$ as in Remark~\ref{rk:paramcl} 
(we choose the notation such that $s_1=s$).
Let $S_{\alpha}$ be the set of 
constituents of $\rho_{s_{\alpha}}$ and 
$\nu_{\alpha}\in\Irr(\Zalg^F)$ the character satisfying
$\cyc{\Res_{\Zalg^F}^{\Galg^F}(\chi),\nu_{\alpha}}_{\Zalg^F}\neq 0$ for all
$\chi\in S_{\alpha}$.
Then we have $\nu_{\alpha}\neq\nu_{\alpha'}$ for $\alpha\neq \alpha'$.
\label{resserie}
\end{lemma}

\begin{proof}
Since $d=|D|$ is prime, we have $|A_{\Galg^*}(s)|=|\Zalg^F|=d$;
see \cite[p.~166]{steinberg} and \cite[II.4.4]{Borel}. 
So, $|A_{\Galg^*}(s)^{F^*}|=d$ implies
$A_{\Galg^*}(s)^{F^*}=A_{\Galg^*}(s)$. Following~\cite[Subsection~8]{BonnafeSLn},  
there is a group homomorphism
$$\omega_s^1:H^1(F^*,A_{\Galg^*}(s))\rightarrow \Irr(\Zalg^F).$$
Again, since $d$ is a prime $\omega_s^1$ is an isomorphism.
Hence, \cite[9.14]{BonnafeSLn} implies
$\nu_{\alpha}\nu_{1}^{-1}=\omega_s^1(\alpha)$ for all 
$\alpha \in H^1(F^*,A_{\Galg^*}(s))$. Since $\omega_s^1$ is
injective (even bijective), the result follows.
\end{proof}

\begin{lemma}\label{nuG}
Suppose that $\Galg$ is simple and $d=|D|$ a prime number. Assume that
$p$ is a good prime for $\Galg$, and fix a positive integer $m$ and 
an $F$-stable character
$\mu\in\Irr(\Zalg^{F^m})$. For $i\in\{1,d\}$, let 
$$N_{i,m}(\mu)=\{\rho_s\,|\, s \in\cal S^{[m]},
\cyc{\rho_s,\rho_s}_{\Galg^{F^m}}=i,
\cyc{\Res_{\Zalg^{F^m}}^{\Galg^{F^m}}(\rho_{s}),\mu}_{\Zalg^{F^m}}\neq 0\},$$
where ${\cal S}^{[m]}$ is a set of representatives for the
semisimple conjugacy classes of $\Galg^{*F^{*m}}$.
Since $\Zalg^F \in \{1,d\}$, we have the following two cases:
\begin{itemize}
\item Suppose $N_{F^m/F}:\Zalg^{F^m}\rightarrow\Zalg^F$ is trivial. Then 
 $N_{1,m}(\mu)^F=\emptyset$ for $\mu\neq 1_{\Zalg^{F^m}}$ and
$N_{1,m}(1_{\Zalg^{F^m}})^F=\{\rho_s\,|\, s\in\cal S,\,
F^*([s]_{\Galg^{*F^{*m}}})=[s]_{\Galg^{*F^{*m}}},
|A_{\Galg^*}(s)|=1\}$, where $\cal S$ is a set of representatives for
the semisimple classes of~$\Galg^{*F^*}$. 
Moreover, for $\mu,\,\mu'\in\Irr(\Zalg^{F^m})$,
one has $|N_{d,m}(\mu)^F|=|N_{d,m}(\mu')^F|$.
\item Suppose $N_{F^m/F}:\Zalg^{F^m}\rightarrow\Zalg^F$ is surjective.
Then for every $\mu,\,\mu'\in\Irr(\Zalg^{F^m})$ and $i\in\{1,d\}$, one
has $|N_{i,m}(\mu)^F|=|N_{i,m}(\mu')^F|$.
\end{itemize}
\end{lemma}
\begin{proof}
Let $t$ be a semisimple element of $\Galg^{*F^{*m}}$ such that
$F^*([t]_{\Galg^{*F^{*m}}})=[t]_{\Galg^{*F^{*m}}}$. Then one
has $F^*([t]_{\Galg^*})=[t]_{\Galg^*}$ and by the Lang-Steinberg
theorem, $[t]_{\Galg^*}$ contains an $F^*$-stable element $s$. As in
Remark~\ref{rk:paramcl}, let $\{s_\alpha \, | \, \alpha\in
H^1(F^{*m},A_{\Galg^*}(s))\}$ be a set of
representatives for the $F^{*m}$-rational classes of $s$. In particular, one
has $[t]_{\Galg^{*F^{*m}}}=[s_{\alpha}]_{\Galg^{*F^{*m}}}$ for some
$\alpha\in H^1(F^{*m},A_{\Galg^*}(s))$.

Suppose $i=d$ and let
$C_s=\{\rho_{s_{\alpha}}\,|\,\cyc{\rho_{s_\alpha},\rho_{s_{\alpha}}}_{\Galg^{F^m}}=d,\,\alpha\in
H^1(F^{*m},A_{\Galg^*}(s))\}$. Therefore, Lemma~\ref{resserie} implies
$$|N_{d,m}(\mu)\cap C_s|=1.$$
We then deduce that
the number $|N_{d,m}(\mu)^F|$ is equal to the number of $F^*$-stable
geometric semisimple classes in $\Galg^*$ whose centralizer is
disconnected. In particular, this number does not depend on $\mu$.
Hence, for all $F$-stable 
$\mu,\,\mu'\in\Irr(\Zalg^{F^m})$, one has
$$|N_{d,m}(\mu)^F|=|N_{d,m}(\mu')^F|.$$

Suppose now that $i=1$.
Note that if $\rho_s\in N_{1,m}(\mu)^F$, the preceding discussion
implies that we can suppose $F^*(s)=s$.
Fix a maximal 
$F^{*m}$-stable torus $\Talg^*$ containing $s$ 
and $(\Talg,\theta_s)$ a pair dual
to $(\Talg^*,s)$. Since $F^*(s)=s$, the character $\theta_s$ of
$\Talg^{F^m}$ is $F$-stable. Moreover,  
Lemma~\ref{surjnorm} implies (because $\Talg$ is connected)
that there is $\theta\in\Irr(\Talg^F)$
satisfying $\theta_s=N_{F^m/F}^*(\theta)$.
Now, if $N_{F^m/F}:\Zalg^{F^m}\rightarrow\Zalg^F$ is trivial, then for
every $z\in\Zalg^{F^m}$, one has $\theta_s(z)=\theta(1)=1$. In
particular,
$\Res_{\Zalg^{F^m}}^{\Talg^{F^m}}(\theta_s)=1_{\Zalg^{F^m}}$, implying
%
$\rho_s$ is over $1_{\Zalg^{F^m}}$. Thus
$\rho_s\in N_{1,m}(1_{\Zalg^{F^m}})^F$.
If $N_{F^m/F}:\Zalg^{F^m}\rightarrow\Zalg^F$ is surjective, then by
Lemma~\ref{surjnorm}, there is $\mu_0\in\Irr(\Zalg^F)$ such that
$\mu= N_{F^m/F}^*(\mu_0)$. Lemma~\ref{resZ} then
implies that $|N_{1,m}(\mu)^F|$ equals the
number $N$ of semisimple characters $\rho_s'$ of~$\Galg^F$
such that the centralizer of $s$ in $\Galg^*$ is connected and
$\cyc{\Res_{\Zalg^F}^{\Galg^F}(\rho_{s}'),\mu_0}_{\Zalg^F}\neq 0$.
In the proof of~\cite[6.6]{Br8}, it is shown that
$$N=\frac{1}{|\Zalg^F|}\cyc{\Gamma_{z},\Gamma_{z'}}_{\Galg^F},$$
for some fixed $z \neq z'\in H^1(F,\Zalg)$ not depending on
$\mu_0$ (note that the part of~the proof of~\cite[6.6]{Br8} which 
we use requires $p$ to be good). 
So, $N$ does not depend on~$\mu$ and the result follows. 
\end{proof}

\subsection{Central characters and automorphisms of $\Balg^F$}

In this subsection, we study the action of the field automorphisms on
the set of $D$-orbits on $\Irr_{p'}(\Balg^F|\nu)$. In fact, we
consider the action of $F$ on the set of $D$-orbits on
$\Irr_{p'}(\Balg^{F^m}|\mu)$ for positive integers $m$ and
$\mu \in \Irr(\Zalg^{F^m})$. 

\begin{lemma}\label{nuB}
Suppose that $\Galg$ is simple and $d=|D|$ is a prime number.
Fix a positive integer $m$ and an $F$-stable character
$\mu\in\Irr(\Zalg^{F^m})$. 
For $i\in\{1,d\}$ 
let $N_{i,m}'(\mu)$ be the set of $D$-orbits on
$\Irr_{p'}(\Balg_0^F|\mu)$ of size $i$.
Then we have the following two cases:
\begin{itemize}
\item Suppose $N_{F^m/F}:\Zalg^{F^m}\rightarrow\Zalg^F$ is trivial. Then 
 $N_{1,m}'(\mu)^F=\emptyset$ for $\mu\neq 1_{\Zalg^{F^m}}$ and
$N_{1,m}'(1_{\Zalg})^{F^m}=\{\chi_{J,1,\psi}\,|\,\psi\in\Irr(\Talg_J^{F^m}),\,
F(\psi)=\psi,\, |H_J^{F^m}|=1\},$
where $\chi_{J,1,\psi}$ is the irreducible character of $\Balg_0^{F^m}$
defined in Lemma~\ref{charB}.
Moreover, for $\mu,\,\mu'\in\Irr(\Zalg^{F^m})$,
one has $|N_{d,m}'(\mu)^F|=|N_{d,m}'(\mu')^F|$.
\item Suppose $N_{F^m/F}:\Zalg^{F^m}\rightarrow\Zalg^F$ is surjective.
Then for every $\mu,\,\mu'\in\Irr(\Zalg^{F^m})$ and $i\in\{1,d\}$, one
has $|N_{i,m}'(\mu)^F|=|N_{i,m}'(\mu')^F|$.
\end{itemize}
\end{lemma}

\begin{proof}
Recall that the set of characters
$\chi_{J,j,\psi}$ for $1\leq j\leq i=i(J)$ is a $D$-orbit of
$\Irr_{p'}(\Balg_0^{F^m})$
and 
\begin{equation}\label{eq:dec}
\Talg_J^{F^m}=\Talg_J^{\circ F^m}\times H_J^{F^m}.
\end{equation}
Let $\Lalg_J$ be the standard Levi subgroup of $\Galg^F$ corresponding
to $J$. Then thanks to Remark~\ref{rk:centrelevi},
the group $H_J$ is isomorphic to the component group $\cal
Z(\Lalg_J)$ and these two
groups are $F$-equivalent. Recall that the inclusion
$\Zalg\subseteq Z(\Lalg)$ induces an $F$-equivariant
surjective homomorphism $h_{\Lalg_J}:\cal Z(\Galg)\rightarrow\cal
Z(\Lalg_J)$, see \cite[4.2]{BonnafeSLn}. 
Moreover, $\cal Z(\Galg)=\Zalg$ because $\Galg$ is simple. 
Then $h_{\Lalg_J}$ is the
projection of $\Zalg$ on $H_J$ in the direct product 
$\Talg_J = \Talg_J^\circ \times H_J$.
Write $\Irr_F(H_J^{F^m})$ for the set of $F$-stable linear characters of
$H_J^{F^m}$.
Let $\psi$ be a linear character of
$\Talg_J^{F^m}$. Then
there are $\psi^{\circ}\in\Irr(\Talg_J^{\circ F^m})$ and $\psi_{H_J}\in
\Irr(H_J^{F^m})$ satisfying $\psi=\psi^{\circ}\otimes\psi_{H_J}$. Let 
$\psi'=\psi^{\circ}\otimes\psi'_{H_J}$ be such that
$$\Res^{\Talg_J^{F^m}}_{\Zalg^{F^m}}(\psi)=\Res^{\Talg_J^{F^m}}_{\Zalg^{F^m}}(\psi').$$
Now, for $h\in H_J^{F^m}$ there is $z_h\in\Zalg^{F^m}$ with
$h_{\Lalg_J}(z_h)=h$ (this is a consequence of the fact that
$h_{\Lalg_J}$ is surjective and $H_J^{F^m}$ equals $\{1\}$ or $H_J$).
Hence, there is
$t\in\Talg_J^{\circ F^m}$ satisfying $z_h=th$. Since
$\psi(z_h)=\psi'(z_h)$, it follows
$$\psi^{\circ}(t)\psi_{H_J}(h)=\psi^{\circ}(t)\psi'_{H_J}(h).$$
Hence, one has $\psi_{H_J}(h)=\psi'_{H_J}(h)$ for all $h\in H_J^{F^m}$
implying
$\psi_{H_J}=\psi_{H_J}'$.
We then deduce from this discussion that the characters
$\psi_{\eta}=\psi^{\circ}\otimes\eta$ for $\eta\in \Irr(H_J^{F^m})$
have distinct restrictions on $\Zalg^{F^m}$.

Let $N_{i,m}'$ be the set of $D$-orbits on $\Irr(\Balg_0^{F^m})$
of size $i$. Note that Remark~\ref{rk:choix} implies $i=|H_J^{F^m}|$,
because $|H_J^{F^m}|=|\cal
Z(\Lalg_J^{F^m})|=|H^1(F^{m},Z(\Lalg_J))|$.
Then  $N_{i,m}'$ is the set of $D_{J,\psi,i}$ for 
$J\subseteq\Delta$ with
$|H_J^{F^m}|=i$ and $\psi\in\Irr(\Talg_J^{F^m})$. 
Moreover, by
Lemma~\ref{la:FrobchrB},
$D_{J,\psi,i}$ is $F$-stable if and only if $F(J)=J$ and $F(\psi)=\psi$.
Now, let $\psi\in\Irr(\Talg_J^{F^m})$ be $F$-stable. If we write
$\psi=\psi^{\circ}\otimes\psi_{H_J}$ as above, then
$F(\psi^{\circ})=\psi^{\circ}$ and $F(\psi_{H_J})=\psi_{H_J}$, implying
$$N_{i,m}'^F=\{D_{J,\psi^{\circ}\otimes \eta,i}\in
N_{i,m}'\,|\, J \subseteq \Delta, \, F(\psi^{\circ})=\psi^{\circ},\,\eta\in
\Irr_F(H_J^{F^m})\}.$$

Suppose $i=d$. In particular, this implies $|H_J^{F^m}|=d$.
Then for $J \subseteq \Delta$ and $\psi^{\circ}\in\Talg_J^{\circ
F^m}$ with $F(\psi^{\circ})=\psi^{\circ}$, the $F$-stable $D$-orbits
$D_{J,\psi^{\circ}\otimes \eta,d}$ for $\eta\in\Irr_F(H_J^{F^m})$
are over distinct $F$-stable characters of $\Zalg^{F^m}$.
For all
$\mu,\,\mu'\in\Irr_F(\Zalg^{F^m})$, we deduce 
$$|N_{d,m}'(\mu)^F|=|N_{d,m}'(\mu')^F|.$$

Suppose $i=1$. Then $H_J$ is trivial and $\Talg_J$ is connected. Using
Lemma~\ref{surjnorm}, the $F$-stable characters of $\Talg_J^{F^m}$ are
the character $\psi=\psi_0\circ N_{F^m/F}$ for any $\psi_0\in
\Talg_J^F$. If $N_{F^m/F}:\Zalg^{F^m}\rightarrow \Zalg^F$ is
trivial, then $N_{1,m}'(1_{\Zalg^{F^m}})^F=N_{1,m}'^F$. In particular, for every
$\mu\neq 1_{\Zalg^{F^m}}$ one has $N_{1,m}'(\mu)^F=\emptyset$. If
$N_{F^m/F}:\Zalg^{F^m}\rightarrow\Zalg^{F}$ is surjective, then
Lemma~\ref{surjnorm} implies that every character $\mu\in\Irr_F(\Zalg^{F^m})$ has
the form $\mu_0\circ N_{F^m/F}$ for some $\mu_0\in \Irr(\Zalg^F)$ and
$$\cyc{\Res_{\Zalg^{F^m}}^{\Talg_J^{F^m}}(\psi),\mu}_{\Zalg^{F^m}}=
\cyc{\Res_{\Zalg^F}^{\Talg_J^F}(\psi_0),\mu_0}_{\Zalg^F}.$$
Note that $\psi_0$ lies in $N'_{1,1}$ because $H_J$
is trivial.
This implies $|N_{1,m}'(\mu)^F|=|N_{1,1}'(\mu_0)|$. Furthermore, the
discussion at the beginning of the proof implies that the numbers
$|N_{d,1}'(\mu_0)|$ do not depend on
$\mu_0\in\Irr(\Zalg^F)$. In particular, for every
$\mu_0\in\Irr(\Zalg^F)$ one has
$$|N_{d,1}'(\mu_0)|=\frac{1}{d}|N_{d,1}'|.$$
Moreover, one has
$|N_{1,1}'(\mu_0)|=|\Irr_{p'}(\Balg_0^F|\mu_0)|-d|N_{d,1}'(\mu_0)|$. Using
Lemma~\ref{la:IrrPprimeBnu}, we deduce that
$|N_{1,1}'(\mu_0)|$ does not depend on $\mu_0$, proving, for all
$\mu,\,\mu'\in\Irr_F(\Zalg^{F^m})$
$$|N_{1,m}'(\mu)^F|=|N_{1,m}'(\mu')^F|.$$
This yields the claim.
\end{proof}

\subsection{An $A$-equivariant bijection respecting central characters}

We keep the above notation and suppose that Hypothesis~\ref{hypoG}
holds. In particular, let $d$ be the order of $D$.
For $\nu\in\Irr(\Zalg^F)$, we write
$A_{\nu}$ for the subgroup of $A$ defined in equation~(\ref{eq:Anu}).
In this subsection, we will give conditions to find an $A_{\nu}$-equivariant
bijection between $\Irr_{p'}(\Galg^F|\nu)$ and $\Irr_{p'}(\Balg^F|\nu)$.

As before, let $\cal O_{D,\nu}$ be the set of $D$-orbits on
$\Irr_{p'}(\Galg^F|\nu)$ and $\cal O'_{D,\nu}$ be the set of $D$-orbits
on $\Irr_{p'}(\Balg^F|\nu)$. 
Furthermore, for $i\in\{1,d\}$, we define $N_i(\nu)$
(resp. $N'_i(\nu)$) to be the set of $D$-orbits on
$\Irr_{p'}(\Galg^F|\nu)$ (resp. $\Irr_{p'}(\Balg^F|\nu)$) of size $i$.   

\begin{proposition}\label{Deq}
Suppose
that $(\Galg,F)$ satisfies Hypothesis~\ref{hypoG}.
Let $\nu\in\Irr(\Zalg^F)$.
If there are $K_{\nu}$-equivariant bijections
$f_i:N_i(\nu)\rightarrow N_i'(\nu)$ for $i\in\{1,d\}$,
then there is an $A_{\nu}$-equivariant bijection between
$\Irr_{p'}(\Galg^F|\nu)$ 
and $\Irr_{p'}(\Balg^F|\nu)$.
\end{proposition}

\begin{proof}We define
$\eta:\mathcal{O}_{D,\nu}\rightarrow\mathcal{O}'_{D,\nu}$ by
$\eta(\omega)=f_1(\omega)$ if $\omega\in N_1(\nu)$, and
$\eta(\omega)=f_d(\omega)$ if $\omega\in N_d(\nu)$. 
Since $\mathcal{O}_{D,\nu}=N_1(\nu)\sqcup
N_d(\nu)$ and $\mathcal{O}_{D,\nu}'=N_1'(\nu)\sqcup N_d'(\nu)$
 and $f_i:N_i(\nu)\rightarrow N'_i(\nu)$
are $K_{\nu}$-equivariant bijections, we deduce that $\eta$ is a
$K_{\nu}$-equivariant bijection between $\mathcal{O}_{D,\nu}$ 
and $\mathcal{O}_{D,\nu}'$.
Write $\Omega=\{\Omega_i\ |\ i\in I\}$ (resp. $\Omega'$) for the set of 
$K_{\nu}$-orbits on $\mathcal{O}_{D,\nu}$ (resp. $\mathcal{O}_{D,\nu}'$).
For every $\Omega_i\in\Omega$, 
we fix $\omega_i\in\Omega_i$. Then $\eta(\omega_i)$ is a set of
representative for $\Omega'$, because $\eta$ is an $K_{\nu}$-equivariant
bijection.
If $|\Omega_i|>1$, we choose any character
$\chi_{\omega_i}\in\omega_i$ and any character
$\chi_{\omega_i}'\in\eta(\omega_i)$. 

Let $\gamma$ be a generator of the cyclic group $K_{\nu}$. 
Note that $\gamma$ is the restriction of a Frobenius map
$F':\Galg\rightarrow\Galg$.
Now, suppose that $|\Omega_i|=1$ (resp. $|\Omega_i'|=1$),
i.e. $\omega_i$ is fixed by $\gamma$.
Recall that
there is a semisimple element $s$ of $\Galg^{*F^*}$ such that $\omega_i$
is the set of constituents of $\rho_s$.
Since $F'(\omega_i)=\omega_i$, we deduce that $s$ and $F'^*(s)$ are
$\Galg^{*F^*}$-conjugate. In particular, Theorem~\ref{imageF}
implies that the constituent $\rho_{s,1}$ of $\rho_s$ is $F'$-stable.
We put $\rho_{s,1}=\chi_{\omega_i}$.  
Moreover, $\eta(\omega_i)$ is fixed by $\gamma$.
Furthermore, following the notation of Lemma~\ref{deforb}, one has
$\eta(\omega_i)=D_{J,\psi,k}$ for some $J\subseteq \Delta$,
$\psi\in\Irr(\Talg_J^F)$ and $k\in\{1,\,d\}$.
Using Lemma~\ref{la:FrobchrB}, $F'(D_{J,\psi,k})=D_{J,\psi,k}$ implies
$F'(J)=J$ and $F'(\psi)=\psi$. Thus, the character
$\chi_{J,1,\psi}\in\eta(\omega_i)$ is $F'$-stable and we put
$\chi_{\omega_i}'=\chi_{J,1,\psi}$.

Let~$\chi\in\Irr_{p'}(\Galg^F|\nu)$. We denote by $\omega$ its $D$-orbit.
Then there is a unique $i\in I$ such that $\omega\in\Omega_i$ implying
there exists a unique $l\in\N$ such that $\omega=\gamma^l(\omega_i)$.
Hence for every $\chi\in\Irr_{p'}(\Galg^F)$, there are unique
$i\in I$, $k,l\in\N$ such that
$$\chi=\delta^k\gamma^l(\chi_{\omega_i}).$$
We define
$\Psi:\Irr_{p'}(\Galg^F|\nu)\rightarrow\Irr_{p'}(\Balg^F|\nu)$ by setting 
$$\Psi(\chi)=\delta^kF'^l(\chi_{\omega_i}'),$$ where $i,j,k$ are as
above. Then $\Psi$ is an $A_{\nu}$-equivariant bijection.
\end{proof}

Recall that  the characteristic $p > 0$ of $\F_q$ is a good prime for $\Galg$
if $p$ does not divide the coefficients of the highest root of the
root system associated to $\Galg$ (see \cite[p.~125]{DM}). 

\begin{theorem}
\label{th:equivbij}
Suppose that $p$ is a good prime for $\Galg$ and that $(\Galg,F)$ satisfies
Hypothesis~\ref{hypoG}. We suppose 
the subgroup of field automorphisms of $\Galg^F$ acts
trivially on $\Zalg^F$. Then for $\nu\in\Irr(\Zalg^{F})$,  the sets
$\Irr_{p'}(\Galg^F|\nu)$ and $\Irr_{p'}(\Balg^F|\nu)$ are
$A$-equivalent. 
\end{theorem}

\begin{proof}
Write $q=p^n$. Recall that $K=\cyc{F_0}$ and $F=F_0^n$.
For any divisor $j$ of $n$, we put $F_j=F_0^j$ and
for $i\in\{1,d\}$, let $N_{i,j}$ and $N'_{i,j}$ be the set of
$D$-orbits on $\Irr_{p'}(\Galg^{F_j})$ and $\Irr_{p'}(\Balg^{F_j})$,
respectively.

For $i\in\{1,d\}$ and $j$ a divisor of $n$, 
let $\cal T_i^{[j]}$ be a set of representatives $s$ for the
$F_j^*$-stable geometric semisimple classes of $\Galg^*$, such that
$|A_{\Galg^*}(s)^{F_j^*}|=i$.
By the Lang-Steinberg theorem, we can suppose that the elements of
$\cal T_i^{[j]}$ are $F_j^*$-stable. Let
$$\cal T_{i,j}=\{s\in\cal
T_i^{[n]}\,|\,F_j^*([s]_{\Galg^*})=[s]_{\Galg^*}\}.$$
Remark~\ref{rk:paramcl} implies
$$|N_{i,n}^{F_j}|=\sum_{s\in \cal
T_{i,j}}\,|H^1(F^*,A_{\Galg^*}(s))^{F_j^*}|.$$
Let $s$ be in $\cal T_{i,j}$. Again, by the
Lang-Steinberg theorem, there is $t\in
[s]_{\Galg^*}$ satisfying $F_j^*(t)=t$. Furthermore, $j$ divides $n$
implies $F^*(t)=t$. We can then suppose that every $s\in\cal T_{i,j}$
is $F_j^*$-stable and $F^*$-stable. By assumption, 
$F_0$ acts trivially on $\Zalg^F$. For every $F_j^*$-stable
semisimple element $s\in\Galg^*$, one has
$A_{\Galg^*}(s)=A_{\Galg^*}(s)^{F_j^*}$.
In particular, one has $\cal T_{i,j}=\cal T_i^{[j]}$ and
$|H^{1}(F^*,A_{\Galg^*}(s))^{F_j^*}|=i$ for every $s\in\cal T_{i,j}$. 
Thus, we have
\begin{equation}\label{eq:pro1}
|N_{i,j}|=\sum_{s\in\cal T_i^{[j]}}\,|H^1(F_j^*,A_{\Galg^*}(s))|
=\sum_{s\in\cal T_i^{[j]}}\, i
=\sum_{s\in \cal T_{i,j}}\, i
=|N_{i,n}^{F_j}|.
\end{equation}
By~\cite[1.1, 6.5]{Br8} and \cite[Theorem 1]{lehrer}, the relative
McKay conjecture holds for~$\Galg^{F_j}$ and $\widetilde{\Galg}^{F_j}$
(here we need that $p$ is a good prime for $\Galg$). So, by
Lemma~\ref{la:diagequivcentral}, for every $\mu \in \Irr(\Zalg^{F^n})$,
there is a $D$-equivariant bijection between
$\Irr_{p'}(\Galg^{F_j}|\mu)$ and $\Irr_{p'}(\Balg^{F_j}|\mu)$. So, in
particular, we get
\begin{equation}\label{eq:pro2}
|N_{i,j}|=|N_{i,j}'|.
\end{equation}
Moreover, with the notation of Lemma~\ref{deforb}, for $i\in\{1,d\}$, one has
$$N_{i,n}'^{F_j}=\{D_{J,\psi,i}\,|\,J \subseteq \Delta,\,F_j(\psi)=\psi\}.$$
Recall that $\Talg_J=\Talg_J^{\circ}\times H_J$ and
$F_j$ acts trivially on $H_J^F$
(because the groups $H_J$ and $\cal
Z(\Lalg_J)$ are $F_0$-equivalent, and
if $H_J^{F_0}$ is non trivial, then $H_J^{F_0}=H_J$). 
For $\psi\in\Irr(\Talg_J^{F_j})$, we write
$\psi=\psi^{\circ}\otimes\eta$ for the decomposition of $\psi$ with
respect to the direct product $\Talg_J^{F_j}\times H_J^{F_j}$. Then
$F_j(\psi)=\psi$ if and only if $F_j(\psi^{\circ})=\psi^{\circ}$.
Put
$$\varphi_i:N_{i,j}\rightarrow
N_{i,n}^{F_j},\,D_{J,\psi^{\circ}\otimes\eta,i}\mapsto
D_{J,\psi^{\circ}\circ N_{F/ F_j}\otimes\eta,i}.$$
Note that on the right hand side, the character $\eta$ is considered as an
$F_j$-stable linear character of~$H_J^F$. This is possible because $H_J^F=H_J^{F_j}$.
Since $N_{F/F_j}:\Talg_J^{\circ F}\rightarrow \Talg_J^{\circ F_j}$
induces a bijection between the linear characters of $\Talg_J^{\circ F_j}$
and the $F_j$-stable linear characters of $\Talg_J^F$ (see
Lemma~\ref{surjnorm}), we conclude that $\varphi_i$ is bijective for
$i\in\{1,d\}$.
In particular, one has
\begin{equation}\label{eq:pro4}
|N_{i,n}'^{F_j}|=|N_{i,j}'|.
\end{equation}
From relations~(\ref{eq:pro1}),~(\ref{eq:pro2}) and~(\ref{eq:pro4}), 
we deduce 
$$|N_{i,n}^{F_j}|=|N_{i,n}'^{F_j}|\quad\textrm{for all divisors }j\
\textrm{of }n.$$
Let $\nu\in\Irr(\Zalg^F)$. 
Suppose that $N_{F/F_j}:\Zalg^{F}\rightarrow\Zalg^{F_j}$ is surjective. 
Then by Lemma~\ref{nuG} and~\ref{nuB}, the numbers $|N_{i,n}(\nu)^F|$
and $|N_{i,n}'(\nu)^F|$ do not depend on $\nu$. In particular, 
$|N_{i,n}^{F_j}|=|N_{i,n}'^{F_j}|$ implies 
\begin{equation}
|N_{i,n}(\nu)^{F_j}|=N_{i,n}'(\nu)^{F_j}|,
\label{eq:pro5}
\end{equation}
where $N_{i,j}(\nu)$ and $N'_{i,j}(\nu)$ are the sets of
$D$-orbits on $\Irr_{p'}(\Galg^{F_j}|\nu)$ and $\Irr_{p'}(\Balg^{F_j}|\nu)$,
respectively.
Suppose that $N_{F/F_j}:\Zalg^{F}\rightarrow\Zalg^{F_j}$ is trivial.
Using the same argument, Lemma~\ref{nuG} and~\ref{nuB} imply
$|N_{d,n}(\nu)^{F_j}|=|N_{d,n}'(\nu)^{F_j}|$.
Moreover, for $\nu\neq 1_{\Zalg^F}$, we have
\begin{equation}\label{eq:pro7}
N_{1,n}(\nu)^{F_j}=\emptyset=N_{1,n}(\nu)^{F_j}.
\end{equation}
In particular, equation~(\ref{eq:pro7}) implies
$$|N_{1,n}(1_{\Zalg^F})^{F_j}|=|N_{1,n}^{F_j}|\quad\textrm{and}\quad
|N_{1,n}'(1_{\Zalg^F})^{F_j}|=|N_{1,n}'^{F_j}|.$$
Since $|N_{1,n}^{F_j}|=|N_{1,n}'^{F_j}|$, we have
\begin{equation}
|N_{1,n}(1_{\Zalg^F})^{F_j}|=|N_{1,n}'(1_{\Zalg^F})^{F_j}|
\label{eq:pro6}
\end{equation}
Thus, equations~(\ref{eq:pro5}),~(\ref{eq:pro7}) and~(\ref{eq:pro6})
imply that, for all divisors $j$ of $n$ and all $\nu\in\Irr(\Zalg^F)$,
one has $$|N_{i,n}(\nu)^{F_j}|=|N_{i,n}'(\nu)^{F_j}|.$$
Using~\cite[13.23]{isaacs}, we deduce that the sets $N_{i,n}(\nu)$ and
$N_{i,n}'(\nu)$ are $K$-equivalent for $i\in\{1,d\}$.  We conclude
with Proposition~\ref{Deq}.
\end{proof}

\begin{remark}
\label{rk:good}
Note that, in the proof of Theorem~\ref{th:equivbij}, we need the
assumption that $p$ is a good prime for $\Galg$ only in order to apply
the results of~\cite{Br8}.
\end{remark}

\begin{corollary}\label{cor:equi2}
Suppose $(\Galg,F)$ satisfies Hypothesis~\ref{hypoG} and $d=2$. If $p$
is a good prime for $\Galg$, then
for $\nu\in\Irr(\Zalg^F)$, the sets $\Irr_{p'}(\Galg^F|\nu)$ and
$\Irr_{p'}(\Balg^F|\nu)$ are $A$-equivalent.
\end{corollary}

\begin{proof}
If $d=2$, then every automorphisms of $\Galg^F$ acts trivially on
$\Zalg^F$. The result is then as direct consequence of
Theorem~\ref{th:equivbij}.
\end{proof}


\section{Inductive McKay condition}
\label{indcond}
In this section, we prove our main result, Theorem~\ref{main}. Our
main ingredient will be the equivariant bijection constructed in 
Corollary~\ref{cor:equi2}. The largest part of this section will be
devoted to the verification of the cohomology condition.

\subsection{Setup}\label{proofsetup}
Throughout this whole Section~\ref{indcond}, let $\Galg$ be a simple
simply-connected algebraic group defined over $\F_q$ of characteristic 
$p>0$ with corresponding Frobenius map $F$ such that the
quotient $$X=\Galg^F/Z(\Galg^F)$$ is a finite simple group and $\Galg^F$
is the universal cover of $X$. Let $\Zalg=Z(\Galg)$ be the center of
the algebraic group $\Galg$ and suppose that $|\Zalg|=2$.
As in Section~\ref{grpsetup}, we embed
$\Galg$ in a connected reductive group $\widetilde{\Galg}$ with an
$\F_q$-rational structure obtained by extending $F$ to
$\widetilde{\Galg}$, such that the center of~$\widetilde{\Galg}$ is
connected and the groups $\widetilde{\Galg}$ and $\Galg$ have the same
derived subgroup. We assume that $p$ is a good prime for $\Galg$. 
Fix an $F$-stable maximal torus $\Talg$ of~$\Galg$ contained in an
$F$-stable Borel subgroup $\Balg$ of~$\Galg$ and
write~$\widetilde{\Talg}$ for the unique $F$-stable 
torus of $\widetilde{\Galg}$ containing $\Talg$. Fix an $F$-stable
Borel subgroup $\widetilde{\Balg}$ of $\widetilde{\Galg}$
containing~$\widetilde{\Talg}$ and~$\Balg$. Let $\Ualg$ be the
unipotent radical of $\Balg$. 
As in Section~\ref{partie2}, we put $K=\cyc{F_0}$. 
Let $D$ be the group of diagonal automorphisms of~$\Galg^F$
induced by the action of $\widetilde{\Galg}^F/\Galg^F$ on~$\Galg^F$.
We set $$\mathcal{A}=\widetilde{\Galg}^F\rtimes K.$$ 
We assume that $F$ acts trivially on the Weyl group of $\Galg$. 
Note that our assumptions imply that $|D|=2$ and that $\Galg^F$
has no graph automorphisms. 

\subsection{Extensions of characters}
We assume the setup from Subsection~\ref{proofsetup}.
Fix $\nu\in\Irr(\Zalg^F)$ and let
$\Phi_{\nu}:\Irr_{p'}(\Galg^F|\nu)\rightarrow\Irr_{p'}(\Balg^F|\nu)$
be an $\mathcal A$-equivariant bijection from Corollary~\ref{cor:equi2}.
For a $D$-stable character $\chi\in\Irr_{p'}(\Galg^F|\nu)$, we define
$$G_{\chi}=\widetilde{\Galg}^F\rtimes\operatorname{Stab}_K(\chi)\quad\textrm{and}\quad
G'_{\chi}=\widetilde{\Balg}^F\rtimes\operatorname{Stab}_K(\Phi_{\nu}(\chi)).$$

\begin{lemma}\label{extension}
Let $\chi\in\Irr_{p'}(\Galg^F|\nu)$ be $D$-stable. Then,
there is an extension of $\chi$ to~$G_\chi$ and an extension of
$\Phi_\nu(\chi)$ to $G_\chi'$.
\end{lemma}

\begin{proof}
There is an integer $i>0$ dividing~$|K|$ such that
$\operatorname{Stab}_K(\chi)=\cyc{F_0^i}$. Let $F'=F_{0}^i$ and
$m=|K|/i$. Then $F'^m=F$. 

First, we show that $\chi$ extends to $G_\chi$. Since $\chi$ is
semisimple, there is a semisimple element $s\in\Galg^{*F^*}$
satisfying $\chi=\rho_s$, where~$\rho_s$ is defined in
equation~(\ref{eq:rhos}). So, one has $F'(\rho_s)=\rho_s$ if
and only if $F'(s)$ and~$s$ are conjugate in $\Galg^{*F^*}$ (see
Theorem~\ref{imageF}). Since $\widetilde{\Galg}^F$ stabilizes
$\rho_s$, we have $A_{\Galg^*}(s)^{F^*}=\{1\}$. In particular,
we can suppose that $s$ is $F'$-stable. Choose an $F'$-stable element
$\widetilde{s}\in\widetilde{\Galg}^{*F^*}$ such that
$i^*(\widetilde{s})=s$. Then $\rho_{\widetilde{s}}$ is an extension of
$\rho_s$ and $F'(\rho_{\widetilde{s}})= \rho_{\widetilde{s}}$ because
$F'(\widetilde{s})=\widetilde{s}$ by Theorem~\ref{imageF}.  
Moreover, since $\operatorname{Stab}_K(\rho_s)$ is cyclic,
$\rho_{\widetilde{s}}$ extends to $G_{\chi}$ and so $\chi=\rho_s$
extends to $G_{\chi}$, as required.

Next, we prove that $\Phi_\nu(\chi)$ extends to $G_\chi'$.
In the notation of Lemma~\ref{charB}, we have 
$\Phi_{\nu}(\chi)=\chi_{J,j,\psi}$ for some $J\subseteq \Delta$, some
positive integer $j$ and some $\psi\in\Irr(\Talg_J^F)$, where 
$\Talg_J$ is the stabilizer of $x\in\Ualg_J^*$ in~$\Talg$ 
as in equation~(\ref{eq:Tj}). Since $\chi_{J,j,\psi}$ is
$\widetilde{\Balg}_0^F$-stable, it follows that $i(J)=1$ and $j=1$.   
Let $\widetilde{\Talg}_J$ be the stabilizer of $x\in\Ualg_J^*$ 
in~$\widetilde{\Talg}$ and write $\phi_J$ for the irreducible character
of $\Ualg_J^F$ chosen for the construction of $\chi_{J,1,\psi}$ in
Lemma~\ref{charB}.  
As in equation~(\ref{eq:hat}), we extend 
$\phi_J$ to the characters $\hat{\phi}_J$ and $\check{\phi}_J$ of 
$\Ualg_1^F\rtimes \Talg_J^F$ and
$\Ualg_1^F\rtimes\widetilde{\Talg}_J^F$, respectively.
The groups $\Talg_J$ and $\widetilde{\Talg}_J$ are connected.
Then by~\cite[0.5]{DM}, there is an $F$-stable torus $\Halg$ such that
$\widetilde{\Talg}_J=\Talg_J\times\Halg$. In particular, 
$\widetilde{\Talg}_J^F=\Talg_J^F\times\Halg^F$ and by~\cite[6.17]{isaacs}
$$\Ind_{\Ualg_1^F\rtimes\Talg_J^F}^{\Ualg_1^F\rtimes\widetilde{\Talg}_J^F}
(\hat{\phi}_J\otimes\psi)=
\sum_{\eta\in\Irr(\Halg^F)}\check{\phi}_J\otimes\psi\otimes\eta.$$
Then we deduce
$$\Ind_{\Balg_0^F}^{\widetilde{\Balg}_0^F}(\chi_{J,1,\psi})=
\sum_{\eta\in\Irr(\Halg^F)}\widetilde{\chi}_{J,1,\psi\otimes\eta},$$
where $\widetilde{\chi}_{J,1,\psi\otimes\eta}$ is the irreducible
character of $\widetilde{\Balg}_0^F$ constructed in Lemma~\ref{charB}.
In particular, the characters $\widetilde{\chi}_{J,1,\psi\otimes\eta}$
are the extensions of $\chi_{J,1,\psi}$ to $\widetilde{\Balg}_0^F$.
Because $\Phi_\nu$ is $\mathcal{A}$-equivariant, we have 
$\Stab_K(\chi_{J,1,\psi})=\Stab_K(\Phi_\nu(\chi))=\Stab_K(\chi)=\cyc{F'}$ where $F'^m=F$
as at the beginning of the proof. Since $F'(\chi_{J,1,\psi})=\chi_{J,1,\psi}$,
Lemma~\ref{la:FrobchrB} implies
$F'(\psi)=\psi$. Thus, $\widetilde{\psi}=\psi\otimes 1_{\Halg^F}$ is
$F'$-stable. Hence, the character $\widetilde{\chi}_{J,1,\widetilde{\psi}}$
is an $F'$-stable extension of $\chi_{J,1,\psi}$ to
$\widetilde{\Balg}_0^F$. So, there is an $F'$-stable extension of
$\Phi_\nu(\chi)$ to $\widetilde{\Balg}^F$. Since
$G_\chi'/\widetilde{\Balg}^F$ is cyclic, we get an extension of 
$\Phi_\nu(\chi)$ to $G_\chi'$, and the proof is complete.
\end{proof}

\subsection{Proof of Theorem~\ref{main}}
We assume the setup from Subsection~\ref{proofsetup}. As described
in~\cite[Section~10]{IMN}, for each irreducible character $\chi$ of a
covering group of $X$, we have to construct a group $G_\chi$
satisfying certain conditions. Because $\Zalg^F$ has order $2$, there
is only one non-trivial covering group of $X$, namely
$S=\Galg^F$. Instead of treating the faithful irreducible characters
of $S$ and the irreducible characters of $X$ separately, we consider
the latter as (non-faithful) characters of $S$ and deal with both
cases simultaneously; see also the proof of \cite[(15.3)]{IMN}. In
case $\chi$ is not faithful, all of the groups constructed in the
following have to be considered ``modulo $\ker(\chi)$''. 

Since $\Galg^F$ has no graph automorphisms, 
the conjugation action of $\mathcal{A}$ induces  
the group of automorphisms of $\Galg^F$ stabilizing $\Ualg^F$. 
Fix a linear character $\nu\in\Irr(\Zalg^F)$. Then $\mathcal{A}$ fixes
$\nu$, since $\Zalg^F$ has order $2$. 
By Corollary~\ref{cor:equi2}, there is an $\mathcal A$-equivariant bijection
$\Phi_{\nu}:\Irr_{p'}(\Galg^F|\nu)\rightarrow\Irr_{p'}(\Balg^F|\nu)$.
In particular, we see that the conditions (1),(2),(3),(4)
of \cite[Section~10]{IMN} are satisfied. The dictionary between our
notation and the one in \cite[Section~10]{IMN} is: $\Ualg^F \leftrightarrow Q$,
$\Balg^F \leftrightarrow T$, $\Phi_\nu \leftrightarrow ()^*$ and 
$\mathcal{A} \leftrightarrow A$ (more precisely: $A$ corresponds to
the set of automorphisms induced by the conjugation action of
$\mathcal{A}$ on $\Galg^F$).

Next, we construct the group $G_\chi$. Fix
$\chi\in\Irr_{p'}(\Galg^F|\nu)$. First, suppose that $\chi$ is not 
$D$-stable. In this case, we set
$$G_{\chi}=\Galg^F\rtimes\operatorname{Stab}_{KD}(\chi)\quad\textrm{and}\quad
G'_{\chi}=\Balg^F\rtimes\operatorname{Stab}_{KD}(\Phi_{\nu}(\chi)).$$
We have to verify conditions (5)-(8) of \cite[Section~10]{IMN}. We have
$\Zalg^F \subseteq Z(G_\chi)$ and the stabilizer of $\chi$ in $\mathcal A$
is induced by the conjugation action of $\Nn_{G_{\chi}}(\Ualg^F)$. 
Note that $\Nn_{G_{\chi}}(\Balg^F)=G'_{\chi}$. Moreover, the group
$C=\Cen_{G_{\chi}}(\Galg^F)=\operatorname{Z}(\Galg^F)$
is abelian and the set $\Irr(C|\nu)$ contains the $G_\chi$-invariant
character $\gamma=\nu$. Thus, conditions (5)-(7) are true in this case.
Note that the factor groups $G_\chi/\Galg^F$ and $G_\chi'/\Balg^F$ are
isomorphic to a subgroup of~$K$ and hence are cyclic. So, the cohomology
groups $H^2(G_\chi/\Galg^F, \C^\times)$ and
$H^2(G_\chi'/\Balg^F, \C^\times)$ are trivial, and therefore condition
(8) holds in this case, too.

Now suppose that $\chi\in\Irr_{p'}(\Galg^F|\nu)$ is $D$-stable.
In this case, we set
$$G_{\chi}=\widetilde{\Galg}^F\rtimes\operatorname{Stab}_{K}(\chi)\quad\textrm{and}\quad
G'_{\chi}=\widetilde{\Balg}^F\rtimes\operatorname{Stab}_{K}(\Phi_{\nu}(\chi)).$$
Again, we have to verify conditions (5)-(8) of \cite[Section~10]{IMN}.
We have $\Zalg^F \subseteq Z(G_\chi)$ and the stabilizer of $\chi$ in
$\mathcal A$ is induced by the conjugation action of $\Nn_{G_{\chi}}(\Ualg^F)$. 
Note that $\Nn_{G_{\chi}}(\Balg^F)=G'_{\chi}$. Moreover, the group
$C=\Cen_{G_{\chi}}(\Galg^F)=\operatorname{Z}(\widetilde{\Galg}^F)$
is abelian. So, conditions (5) and (6) of \cite[Section~10]{IMN} are satisfied.

By Lemma~\ref{extension}, there exists an extension $\widetilde{\chi}'$ of
$\chi'=\Phi_\nu(\chi)\in\Irr_{p'}(\Balg^F|\nu)$ to~$G_{\chi}'$.
For abbreviation, we set $M=\Balg^F$.
Since $M\subseteq MC\subseteq G_{\chi}'$, the character
$\Res_{MC}^{G_{\chi}'}(\widetilde{\chi}')$ is irreducible. Since
$MC$ is a central product, there are
$\chi_{M}'\in\Irr(M)$ and $\chi_C'\in\Irr(C)$ satisfying 
\begin{equation}
\label{eq:res}
\Res_{MC}^{G_{\chi}'}(\widetilde{\chi}')=\chi_{M}'\cprod\chi_C',
\end{equation}
see Subsection~\ref{cenprod}. Furthermore, for $g\in M$, one has
$\Res_{MC}^{G_{\chi}'}(\widetilde{\chi}')(g)=\chi'(g)$, so
$\chi_M'=\chi'$. Moreover, since $\chi_M$ and $\widetilde{\chi}'$ are
$G_{\chi}'$-stable, it follows from equation~(\ref{eq:res}) that
$\chi_C'$ is $G_{\chi}'$-stable. Note that
$$\Res_{\operatorname{Z}(M)}^{G_{\chi}'}(\widetilde{\chi}')=
\Res_{\operatorname{Z}(M)}^M(\chi')=\chi'(1)\cdot\nu.$$
Hence, equation~(\ref{eq:res}) implies
$\Res_{\operatorname{Z}(S)}^{C}(\chi_C')=\nu$. We put
\begin{equation}
\label{eq:gamma}
\gamma=\chi_C'.
\end{equation} 
It follows that $\gamma$ is a $G_{\chi}'$-stable and hence
$G_\chi$-stable element of $\Irr(C|\nu)$. So, condition (7)
of \cite[Section~10]{IMN} holds in this situation.

We choose a positive integer $i$ dividing~$|K|$ such that
$\operatorname{Stab}_K(\chi)=\cyc{F_0^i}$. Let $F'=F_{0}^i$ and
$m=|K|/i$. Then $F'^m=F$. Put
$$E(\chi)=\{\vartheta\in\Irr(\widetilde{\Galg}^F)\ |\
\Res_{\Galg^F}^{\widetilde{\Galg}^F}(\vartheta)=\chi\}.$$
Since the semisimple character $\chi$ is $F'$-stable, there is an
$F'$-stable semisimple element $s\in\Galg^{*F^*}$ satisfying $\chi=\rho_s$. 
Choose an $F'$-stable semisimple element
$\widetilde{s}\in\widetilde{\Galg}^{*F^*}$ such that
$i^*(\widetilde{s})=s$. Then,
$$E(\chi)=\{\rho_{\widetilde{s}t}\ |\
t\in\operatorname{Z}(\widetilde{\Galg}^*)^{F^*}\}.$$
In particular, $|E(\chi)|=|\operatorname{Z}(\widetilde{\Galg}^F)|$.
Moreover, note that $\rho_{\widetilde{s}t}$ is $F'$-stable if and only
if $F'^*(t)=t$. In particular, one has
$$|E(\chi)^{F'}|=|\operatorname{Z}(\widetilde{\Galg}^*)^{F'^*}|=p^i-1.$$
Put
$$\quad L(\chi)=\{\varphi\in\Irr(\Galg^FC)\ |\
\Res_{\Galg^F}^{\Galg^FC}(\varphi)=\chi\}.$$
Since $\Galg^FC$ is a central product, every character of $\Galg^F$
extends to $\Galg^FC$. Now, using~\cite[6.17]{isaacs}, one has
$L(\chi)=\{\varphi\otimes\xi\ |\ \xi\in \Galg^FC/\Galg^F\}$. In
particular, $|L(\chi)|=|C|/2$.  Since $\chi$ is $F'$-stable, we have
$\varphi=\chi\otimes\xi\in L(\chi)^{F'}$ if and only if $\xi$ is
$F'$-stable. Because $C=\operatorname{Z}(\widetilde{\Galg}^F)$, it
follows that
\begin{equation}
\label{eq:ordL}
|L(\chi)^{F'}|=\frac{1}{2}\cdot|\operatorname{Z}(\widetilde{\Galg})^{F'}|=\frac
 1 2 (p^i-1)=\frac 1 2 |E(\chi)^{F'}|.
\end{equation}
Consider the map
$$\Theta:E(\chi)^{F'}\rightarrow L(\chi)^{F'},\,\vartheta\mapsto
\Res_{\Galg^FC}^{\widetilde{\Galg}^F}(\vartheta).$$
Note that $\Theta$ is well-defined because
$$F'(\Theta(\vartheta))=\Theta(F'(\vartheta))=\Theta(\vartheta).$$
By~\cite[6.B]{BonnafeSLn}, we know that
$|\widetilde{\Galg}^F/\Galg^FC|=2$. So Clifford
theory (see~\cite[6.19]{isaacs}) and equation~(\ref{eq:ordL}) imply
\begin{equation}
\label{eq:qq}
|\Theta(E(\chi)^{F'})|\ge\frac 1 2 |E(\chi)^{F'}|=|L(\chi)^{F'}|.
\end{equation}
So, $\Theta$ is surjective. In particular, for the
irreducible $F'$-stable character $\gamma\in\Irr(C|\nu)$ defined in
equation~(\ref{eq:gamma}), the character $\chi\cprod\gamma$ (which lies
in $L(\chi)^{F'}$) extends to an $F'$-stable character
$\widetilde{\chi}$ of $\widetilde{\Galg}^F$. So,
by~\cite[11.22]{isaacs}, the character $\widetilde{\chi}$ extends to 
$G_{\chi}$, because  $\Stab_K(\chi)$ is cyclic and stabilizes
$\widetilde{\chi}$. In particular, $\chi\cprod\gamma$ extends to $G_{\chi}$.

So, we have shown that there is a $G_{\chi}$-stable irreducible
character $\gamma$ of $C$ over $\nu$ such that $\chi\cprod\gamma$ and
$\Phi_{\nu}(\chi)\cprod\gamma$ extend to extend to $G_{\chi}$ and
$\Nn_{G_{\chi}}(M)$, respectively. Then, by~\cite[11.7]{isaacs} we have
$$[\chi\cprod\gamma]_{G_{\chi}/\Galg^FC}=
[\Phi_{\nu}(\chi)\cprod\gamma]_{\Nn_{G_{\chi}}(M)/MC}.$$ 
Thus, condition (8) of \cite[Section~10]{IMN} is satisfied, too. Hence, $X$ is
``good'' for $p$ and the proof of Theorem~\ref{main} is complete. \hfill $\square$

\begin{remark} The following can be said about those finite simple
groups of type $B_m$, $C_m$ or $E_7$ which are not included in
Corollary~\ref{cormain}:
\begin{enumerate}
\item[(a)] By~\cite[Remark 2]{Br6} and~\cite{Cabanes}, all simple
groups $B_m(2^n)$ (isomorphic to $C_m(2^n)$) are ``good'' for the
defining characteristic $p=2$.

\item[(b)] Theorem~5 in \cite{Br6} implies that all simple groups
$E_7(2^n)$ are ``good'' for $p=2$. 

\item[(c)] Our methods were not sufficient to show that the simple
groups $E_7(3^n)$ are ``good'' for $p=3$. By Remark~\ref{rk:good}, it
is enough to prove that the following properties hold. First, 
the group $\Galg^F$ (here, $\Galg$ denotes a simple simply-connected
group of type $E_7$ defined over a finite field of characteristic $3$
and $F$ denotes the corresponding Frobenius map) satisfies the relative McKay
conjecture at the prime $p=3$. Second, 
the number of semisimple characters $\rho_s$ of $\Galg^F$
lying over $\nu\in\Irr(\operatorname{Z}(\Galg^F))$ and
such that the group $\Cen_{\Galg^{*F^*}}(s)$ is connected, does not
depend on $\nu$. 
\end{enumerate}
\end{remark}

\section*{Acknowledgements}

We thank G.~Malle for fruitful discussions on the subject, for his
reading of the manuscript and for comments on this paper.


\def\cprime{$'$}

\end{document}